\documentclass[10pt]{article}
\usepackage{amsfonts}
\usepackage{amsmath,amsthm,amscd,amssymb,mathrsfs,setspace, textcomp}
\usepackage{latexsym,epsf,epsfig}
\usepackage{color}
\usepackage[hmargin=2.4cm,vmargin=2.4cm]{geometry}
\usepackage{hyperref}

\newcommand{\ds}{\displaystyle}

\newcommand{\reals}{\mathbb{R}}
\newcommand{\realstwo}{\mathbb{R}^2}
\newcommand{\realsthree}{\mathbb{R}^3}

\newcommand{\xb}{{\bf{x}}}

\newcommand{\cD}{\mathscr{D}}
\newcommand{\Dn}{\partial_{\nu}}

\newcommand{\Dz}{\partial_z}
\newcommand{\cE}{{\mathcal{E}}}
\newcommand{\cF}{\mathcal{F}}

\newcommand{\R}{\mathbb{R}}

\newcommand{\bu}{\mathbf u}
\newcommand{\bz}{\mathbf z}

\newcommand{\bE}{\mathbf{E}}
\newcommand{\bT}{\mathbb{T}}
\newcommand{\lb}{ \langle}
\newcommand{\rb}{ \rangle}

\newcommand{\ou}{\overline u}
\newcommand{\of}{\overline{\phi}}
\theoremstyle{plain}
\newtheorem{theorem}{Theorem}[section]
\newtheorem{lemma}[theorem]{Lemma}

\newtheorem{corollary}[theorem]{Corollary}

\theoremstyle{remark}
\newtheorem{remark}{Remark}[section]
\numberwithin{equation}{section}
\numberwithin{theorem}{section}
\numberwithin{remark}{section}
\numberwithin{assumption}{section}
\numberwithin{condition}{section}

\begin{document}

\title{Feedback stabilization of a fluttering panel in an inviscid subsonic potential flow
\thanks{I. Lasiecka was partially supported by the National Science Foundation with grant NSF-DMS-0606682 and the United States Air Force Office of Scientific Research with grant AFOSR-FA99550-9-1-0459.
 J.T. Webster was partially supported by National Science Foundation with grant NSF-DMS-1504697.
}
}


 \author{ 
            { Irena Lasiecka \footnote{University of Memphis, Memphis,  lasiecka@memphis.edu} }  { Justin T. Webster \footnote{The College of Charleston, Charleston, SC; websterj@cofc.edu}} }


%

\maketitle

\begin{abstract}
Asymptotic-in-time feedback control of a panel interacting with an inviscid, subsonic flow is considered. The classical model \cite{dowellnon} is given by a clamped nonlinear plate strongly coupled to a convected wave equation on the half space. In the absence of imposed energy dissipation the plate dynamics converge to a compact and finite dimensional set \cite{delay,fereisel}. With a sufficiently large velocity feedback control on the structure we show that {\em the full flow-plate system exhibits strong convergence} to the stationary set in the natural energy topology. To accomplish this task, a novel decomposition of the nonlinear plate dynamics is utilized: a smooth component (globally bounded in a higher topology) and a uniformly exponentially decaying component.

Our result implies that {\em flutter} (a periodic or chaotic end behavior) can be eliminated (in subsonic flows) with sufficient frictional  damping in the structure. While such a result has been proved in the past  for regularized plate models (with rotational inertia terms or thermal considerations \cite{springer,conequil1,ryz,ryz2}), this is the first treatment which  does not incorporate smoothing effects for the structure.  

\vskip.2cm
\noindent {\em Key Terms}: strong stability, nonlinear plates, flow-structure interaction, PDE with delay, attractors
\vskip.1cm
\noindent {\em 2010 AMS}: 74F10, 74K20, 76G25, 35B40, 35G25, 37L15
\end{abstract}

\section{Introduction}\label{Intro}

	This work is motivated by physical problems in aeroelasticity. {\em Flutter} of a thin structure immersed in a flow field occurs (as a bifurcation at a particular flow velocity) when the structural modes couple with the dynamic loading of the flow, creating a feedback \cite{bolotin,dowellnon}. The canonical example of {\em panel flutter} occurs in aircraft paneling: panel deformations perturb the flow field, which then determines the dynamic loading on the plate (resulting in structural deformation, etc.). The principal challenge in aeroelasticity is the control, suppression, or prevention of the flutter of mechanical structures in flight \cite{B,dowell1,dowellrecent}. Though the {\em flutter point} may be predicted from linear dynamics, accurately describing post-flutter displacements of a panel requires a nonlinear model \cite{dowellnon,B}.

From a mathematical point of view the flutter problem may be described as {\it stabilization to a given stationary  set} for a coupled evolutionary PDE system: a nonlinear clamped plate interacting with a potential flow \cite{bolotin,dowellnon,HolMar78}.  Here, we consider Berger plates in the absence of {\em rotational inertia} (consistent with the conventions of aeroelasticity \cite{bolotin,dowellnon,B,dowell1} and references therein), where the plate is thin and in-plane accelerations are not accounted for \cite{lagnese}. The task at hand, then, is to determine conditions under which the PDE system is {\it strongly stable}, in the sense that full flow-plate trajectories converge (in the strong topology of the underlying  finite energy space) to a set of interest (e.g., the equilibria set for the dynamical system generated by the solutions). Such a strong convergence to equilibria does not seem possible in the absence of dissipative mechanisms. Thus, we seek an interior velocity feedback control (of sufficient size---depending on the intrinsic parameters of the problem, not on the particular initial data). Under some physical assumptions this will imply that flutter is non-present asymptotically in time. {\em Our result demonstrates that flutter can be eliminated  (in the subsonic regime), provided a sufficiently large damping is present   in the plate.} In agreement with the physics of the problem,  the linear model for the flow dynamics remains undamped.

In this treatment, we primarily focus on the Berger plate model. Other work on flutter dynamics have considered Berger-like beam and plate models (see \cite{Memoires} and references therein), as well as the scalar von Karman equations \cite{springer,delay,supersonic,webster}. Many of our supporting results here are valid for both the von Karman {\em and} Berger plate, but the principal results presented below in Section \ref{mainresults} concern Berger plate dynamics. The Berger plate is a nonlocal, cubic-type nonlinearity. The Berger nonlinearity is arrived at via a simplification of the scalar von Karman equation \cite{berger,inconsistent,studyberger} (with the scalar von Karman equation itself a simplification of the {\em full von Karman system} \cite{koch, lagnese}).\footnote{In fact, in many cases the analysis of the scalar von Karman plate subsumes that of Berger \cite{springer}.} 
However, there are times (such as in the main result presented here) when working in {\em higher topologies} for the dynamics produces a stark difference in the behavior of the von Karman plate and the Berger plate. 

Physically, when the plate has no portion of its boundary {\em free}, the Berger approximation is taken to be valid \cite{inconsistent,studyberger} (and especially in the case of clamped conditions, as we take here). The von Karman plate dynamics were considered in \cite{conequil1} (and in the treatments leading up to \cite{conequil1}, including \cite{chuey,ryz,ryz2} and \cite{springer}). In \cite{conequil1}, large {\em static and viscous} damping was considered to obtain the main result (which is valid for both the non-rotational von Karman and Berger dynamics). Here, we consider the more physical case where {\em only viscous damping is active} in the plate dynamics (though we cannot dispense with the size requirement).

\subsection{Previous Literature in Relation to the Present Work}\label{prevv}

The study in \cite{delay} provides long-time asymptotic  properties of (finite energy) solutions for the model herein (for Berger, von Karman, and other nonlinearities of {\em physical type} \cite{supersonic}) {\em with no restrictions on the damping size}---it can be taken to be zero. In this case, the originally ``rough'' plate dynamics become asymptotically regular and finite dimensional in nature. Thus, the flow has the ability to dissipate mechanical plate energy and thereby induce a degree of stability to the structural dynamics. However, this stability, in general, may be of a chaotic character. It is one of the principal purposes of this treatment to show that, in the subsonic case, sufficiently large dissipation in the structure alone prevents this chaotic behavior and leads to strong stability to the equilibria set. 

We now mention three other closely related scenarios in the literature to date: (i) the addition of thermoelastic dynamics to the nonlinear plate, (ii) the presence of rotational inertia terms {\em and} {\em strong} mechanical damping, and (iii) large {\em static and viscous} damping in the model presented here. (These treatments address von Karman plates, though the results also hold for Berger plates.) The treatments in \cite{ryz,ryz2} consider the plate with the addition of a heat equation in the plate dynamics. In this case no damping is needed, as the analytic smoothing and stability properties of thermoelasticity provide ultimate compactness of the plate dynamics and, furthermore, provide a stabilizing effect for the flow to equilibria (in subsonic flows). As for (ii) mentioned above, results on ultimate compactness of plate dynamics, as well as convergence (for subsonic flows) of full flow-plate trajectories to points of equilibrium, were announced in \cite{chuey}, with a proof appearing in \cite{springer}. The recent treatment \cite{conequil1} addresses (iii); it is shown that by considering ``large" static damping and ``large" viscous damping---of the form $D\cdot [u_t+u]$ in the plate equation, $D$ sufficiently large---flow-plate trajectories with finite energy initial data converge to the set of equilibria. The techniques in \cite{conequil1} are rooted in Lyapunov methods and utilize {\em critically} the prior work in \cite{Memoires,springer,ryz,ryz2}. The key point here is that for large static and viscous damping, exponential decay of plate velocities provides global-in-time Hadmard continuity which permits approximation by smooth data (for which the desired result holds---Theorem \ref{regresult}).

The main goal of the present work is to obtain strong stability results ( strong convergence to a set of equilibria) {\it without accounting for smoothing effects in plate  dynamics}  and without  {\it without assuming large static damping}  imposed on the structure (plate). The primary mathematical issue to contend with is low regularity of the hyperbolic Neumann map (boundary into the interior)---i.e., the failure of the uniform Lopatinski conditions in dimensions higher than one. This precludes direct analysis of the coupling via the compactness properties of the aforementioned Neumann map. (Such an approach was used critically in previous analyses \cite{b-c,LBC96,chuey,springer,ryz,ryz2}.) In fact, while we still rely on these past developments, the key point of departure is that we can no longer afford to treat the problem component-wise. Rather, we must rely on global analysis involving relaxed long time  compactness  \cite{dafermos,slemrod}, which depends on uniform long   time invariance of higher energies \cite{conequil1}. This latter property is highly non-trivial due to the effects of the plate nonlinearity combined with the lack of damping exhibited by the flow.  In order to overcome this obstacle  we shall rely on  the idea introduced in \cite{springer} and used in \cite{delay} where structural equation under the  influence of a flow is 
represented (for large times) by  delay system. It is then shown in \cite{delay} that the corresponding decomposed part admits  a  {\it smooth}  global attracting set. In order to 
extract  a long time  convergence  valid for the {\it entire flow-structure system} it would suffice  to have a control of  $ L_1$ rate of convergence  of delayed plate solutions to such attractor. 
This brings forward a concept of {\it exponential attractor}, which by definition provides exponential   convergence rates, but it may loose "smoothness "  properties.
   The above leads to a well known  dichotomy between smoothness and rates of convergence to attracting sets.  Resolving this dichotomy lies in the heart of the  problem presented. A delicate argument based on  a careful analysis of  delayed dynamics along with  an uniform Hadamard continuity property  exhibited  by of a suitable decomposed  part of  
the full dynamics only,  allows to reconcile this dilemma. 
 As a consequence.  we show that, with a sufficiently large {\em viscous damping parameter}, the entire  flow-plate finite energy dynamics exhibits strong convergence properties. We emphasize that our results require {\em only viscous damping} in the structure and {\em do not make use of any advantageous parabolic (smoothing) effects} or additional static damping.
 One may argue with the requirement of sufficiently large size of the  damping (not needed in parabolic like problems), however the fact that we  de facto  stabilize delayed structure (with inherent instability)  justifies  mathematically \cite{pignotti,incase} a need for large viscous damping.

Our result provides mathematical insight to the panel flutter problem; in \cite{delay} the flow, while driving the plate dynamics, also contributes a stabilizing effect to the plate dynamics. Our present result indicates that for a panel in a subsonic flow, strong stability of the plate (due to mechanical damping) can be {\em transferred} to the flow, in some sense. (We note that in all real panels there is some degree of internal damping.) These results are in agreement with experimental and numerical studies wherein {\em divergence} (or `buckling') of panels is observed for subsonic flows, in contrast to chaotic or periodic end behavior ({\em i.e., flutter}) in the case of supersonic flow velocities.
We quote from a recent survey article authored by foremost aeroelastician E. Dowell \cite{dowellrecent} with regard to subsonic flows: ``...if the trailing edge, as well as the leading edge, of the panel is fixed then divergence (static aeroelastic instability) will occur rather than flutter."
In the analysis below, the subsonic nature of the flow is critical, as it provides a viable energy balance equation; this energy balance is not available for supersonic flows \cite{supersonic}. 

The engineering literature in \cite{berger,inconsistent,studyberger} analyzes the validity of Berger's plate approximation (originally appearing in \cite{berger}) and concludes that it is accurate in the case of clamped and hinged boundary conditions. The approximation is based on the assumption that the second stain invariant is negligible. Mathematically, the Berger dynamics are addressed in \cite{oldchueshov1,Memoires,oldchueshov2} (in the context of aeroelasticity), and also in the papers \cite{pelin1,pelin2}. Much of the abstract work in \cite{springer} (which is focused on the scalar von Karman nonlinearity) also applies to the Berger dynamics. The main results appearing in \cite{delay,fereisel,supersonic,webster} for the model presented above discuss and analyze both the Berger {\em and} von Karman nonlinearities, as well as other ``physical" nonlinearities satisfying a specific set of bounds which are referred to as {\em nonlinearities of physical type}).
\subsection{Outline}
The remainder of the paper is organized as follows:
\vskip.1cm
\noindent [\textsection 1] The remainder of the Introduction is devoted to a discussion of the mathematical model, including the relationship between our analysis and previous literature. We discuss the fundamental notions of well-posedness and associated energies, and describe some key results from \cite{delay,conequil1} which motivate this analysis. We then describe the main results. 
\vskip.1cm
\noindent [\textsection 2] Section \ref{techpres} describes the technical preliminaries needed to read the proof of the main results. This section can also be read as a detailed summary of the preceding work on the specific flow-plate model described herein. 

\vskip.1cm
\noindent [\textsection 3] Section \ref{prevlit} describes the supporting results from \cite{delay}, used here, which deal with the existence of a compact global attractor for the plate dynamics, as well as some weak stability properties.

\vskip.1cm
\noindent [\textsection 4] In Section \ref{highernorms00} we prove the primary {\em supporting} result upon which our later analysis is based: Theorem \ref{highernorms0}, which says that a global-in-time bound on the plate dynamics in higher norms yields the desired convergence to equilibria result for the full flow-plate dynamics. This section outlines a strategy of passing convergence results from {\em smooth} initial data onto finite energy initial data, and we describe in Section \ref{techdiff} the primary technical obstacles in showing the convergence to equilibria result for finite energy initial data in the absence of smooth effects or large static damping. 
\vskip.1cm
\noindent [\textsection 5] Section \ref{proofmain} begins by providing an overview of the main steps in the proof, and their relation to the established results in Section \ref{techpres}. (We believe this to be helpful to the reader, owing to the technical, and multistep nature of the proof.) The proof of the main result follow in substeps: a decomposition of plate solutions (Sections \ref{expsec} and \ref{regsec}), a key decomposition for the full flow-plate dynamics, an analysis of the decomposed dynamics, and a concluding argument which yields the final result.
\vskip.1cm
\noindent [\textsection 6] In Section \ref{openprob} we discuss the open problem of eliminating the size requirements on the damping.
\vskip.1cm
\noindent [\textsection 7] Readers will find the authors' acknowledgements in the final section. 

	\subsection{Notation}
	For the remainder of the text we write $\xb$ for $(x,y,z) \in \realsthree_+$ or $(x,y) \in \Omega \subset \realstwo_{(x,y)}$, as dictated by context. Norms $\|\cdot\|$ are taken to be $L_2(D)$ for a domain $D$. The symbols $\nu$ and $\tau$ will be used to denote the unit normal and tangent vectors to a given domain, again, dictated by context. Inner products in $L_2(\realsthree_+)$ are written $(\cdot,\cdot)$, while inner products in $L_2(\R^2\equiv\partial\R^3_+)$ are written $\left<\cdot,\cdot\right>$. Also, $ H^s(D)$ will denote the Sobolev space of order $s$, defined on a domain $D$, with $H^s_0(D)$ denoting the closure of $C_0^{\infty}(D)$ in the $H^s(D)$ norm
(which we denote by $\|\cdot\|_{H^s(D)}$ or $\|\cdot\|_{s,D}$).  When $s =0 $ we may abbreviate the notation
to $\| \cdot \| $. We will also utilize the homogenous Sobolev space $$W^k(\realsthree_+) \equiv \left\{\phi(\xb) \in L_2^{loc}(\realsthree_+) ~:~ \|\phi\|^2_{W^k}\equiv \sum_{j=0}^{k-1}\|\nabla \phi \|^2_{j,\realsthree_+} \right\}.$$ We make use of the standard notation for the trace of functions defined on $\realsthree_+$, i.e., for $\phi \in H^1(\realsthree_+)$, $tr[\phi]=\phi \big|_{z=0}$ is the restriction of $\phi$ on the plane $\{\xb:z=0\}$. (We use analogous notation for $tr[w]$ as the map from $H^1(\Omega)$ to $H^{1/2}(\partial \Omega)$.) Finally, we will denote a ball of radius $R$ centered at $c$ in a topological vector space $X$ by $B_X(c,R)$; when the center is the zero element we simply write $B_X(R)$.

\subsection{Mathematical Model and Natural Energies}
	The gas-flow environment is modeled by $\realsthree_+=\{(x,y,z): z > 0\}$. The plate
 is  immersed in an inviscid potential flow (over body) with velocity $U<1$ in the  $x$-direction. (Here we (i) normalize $U=1$ to be Mach 1, i.e., $0 \le U <1$ corresponds to a subsonic flow, and (ii) consider only the dynamics in the upper half space.) The panel is modeled by a bounded domain $\Omega \subset \reals^2_{(x,y)}=\{\xb: z = 0\}$ with smooth boundary $\partial \Omega = \Gamma$, and
the scalar function $u: \Omega \times \R_+ \to \reals$ represents the (transverse) displacement of the plate in the $z$-direction at $(x,y)$ at the moment $t$. For the flow component of the model we make use of linear
 potential theory \cite{bolotin,dowell1}, and the (perturbed) flow potential is given by $\phi:\realsthree_+\times \reals_+ \rightarrow \reals$. The strong coupling occurs (i) in the dynamic pressure term (RHS of the plate equation), which contains the acceleration potential of the flow, and (ii) in the downwash of the flow (Neumann condition), the latter including the material derivative of the structure.
	\begin{equation}\label{flowplate}\begin{cases}
u_{tt}+\Delta^2u+ku_t+f(u)= p_0+r_{\Omega}tr[\big(\partial_t+U\partial_x\big)\phi]& \text { in } \Omega\times (0,T),\\
u(0)=u_0;~~u_t(0)=u_1,\\
u=\Dn u = 0 & \text{ on } \Gamma\times (0,T),\\
(\partial_t+U\partial_x)^2\phi=\Delta \phi & \text { in } \realsthree_+ \times (0,T),\\
\phi(0)=\phi_0;~~\phi_t(0)=\phi_1,\\
\partial_z \phi = \big[(\partial_t+U\partial_x)u \big]_{\text{ext}} & \text{ on } \realstwo_{(x,y)} \times (0,T).
\end{cases}
\end{equation}
The notation $r_{\Omega}$ corresponds to the restriction (of the function's domain) to $\Omega$ (in the appropriate functional sense) for functions supported on $\partial \mathbb R_+^3$; correspondingly, we have the extension by zero for functions supported on $\Omega$ to the entirely of $\partial \mathbb R^3_+$ (in particular, for functions in $H_0^2(\Omega)$), denoted by $u_{\text{ext}}$. The term $p_0$ is a static pressure on on the plate. The nonlinearity of principal interest here is the Berger nonlinearity: 
\begin{equation}\label{berger}
f_B(u) = [b - ||\nabla u||^2]\Delta u.
\end{equation}
The parameter $b\ge 0$ is a physical parameter \cite{inconsistent,studyberger} which corresponds to in-plane bending or stretching.\footnote{In full generality, $b$ can be any real number; we choose the more mathematically interesting---non-dissipative---case, with $b\ge 0$.}

The  {\it plate energy} is defined as usual \cite{springer,lagnese}:
\begin{align}\label{plateenergy}
E_{pl}(u) =& \dfrac{1}{2}\big[\|u_t\|_{\Omega}^2+ \|\Delta u\|_{\Omega}^2 \big] +\Pi(u).
\end{align}
\noindent $\Pi(u)$ is a potential of the nonlinear and nonconservative forces,  given
by
 \begin{equation}\label{Pi}
 \Pi(u)=\Pi_B(u)=\dfrac{1}{4}||\nabla u||^4_{\Omega}-\dfrac{b}{2} ||\nabla u||^2_{\Omega}-\lb p_0,u\rb_{\Omega}.
 \end{equation}
 The natural energies associated with the subsonic wave and interactive dynamics are given below:
\begin{align}
 E_{fl}(\phi) = & ~\dfrac{1}{2}\big[\|\phi_t\|_{\R^3_+}^2-U^2\|\partial_x\phi\|_{\R^3_+}^2+\|\nabla \phi\|_{\R^3_+}^2\big];\label{flowenergy}\\
 E_{int}(u,\phi)=&~2U\lb tr[\phi],u_x\rb_{\Omega}\label{intenergy}.
\end{align}
 The total (unsigned) energy is then defined to be \begin{equation}\label{totalenergy}
\mathcal E(u(t),\phi(t))=\mathcal E(t) = E_{pl}(u(t))+E_{fl}(\phi(t))+E_{int}(u(t),\phi(t)).
\end{equation}
We will also need to consider {\em positive energies}, so we define
\begin{align}\label{posen}
\Pi_*(u)=\dfrac{1}{4}||\nabla u||^4,&~~~
E_*(u) =  \dfrac{1}{2}[||u_t||^2+||\Delta u||^2]+\Pi_*(u),\\
\mathcal E_*(u,\phi) =& E_*(u)+E_{fl}(\phi).
\end{align}

According to these norms, the natural energy space is then: 
\begin{equation}\label{energyspace1}
Y = Y_{pl} \times Y_{fl}\equiv \left(H_0^2(\Omega)\times L_2(\Omega)\right)\times\left(W^1(\realsthree_+) \times L_2(\realsthree_+)\right),
\end{equation}
taken with norm
\begin{equation}
||(u,v;\phi,\psi)||^2_Y = ||\Delta u||_{\Omega}^2+||v||_{\Omega}^2+||\nabla \phi||^2_{\realsthree_+}+||\psi||_{\realsthree_+}^2.
\end{equation}
We will also consider a stronger space:
\begin{equation}\label{energyspace2}
Y_s\equiv \left(H_0^2(\Omega)\times L_2(\Omega)\right)\times\left(H^1(\realsthree_+) \times L_2(\realsthree_+)\right).
\end{equation}
\subsubsection{Damping} The damping coefficient is given by $k \ge 0$, representing viscous damping (or dissipation) in the model. The can be viewed as intrinsic damping for the structure, or imposed as a velocity feedback control. For certain results (see \cite{delay}) we can consider $k=0$, however, for the main results below we will need specific assumptions on the size of the damping coefficient $k$.

Additionally, in some auxiliary results presented below (especially those related to \cite{conequil1}) we will refer to {\em static damping}---this occurs when the viscous damping $ku_t$ is accompanied by a term of the form $\beta u$. The resulting system, which will be referred to below, is 
\begin{equation}\label{withbeta}\begin{cases}
u_{tt}+\Delta^2u+\boxed{ku_t+\beta u}+f(u)= p_0+r_{\Omega}tr[\big(\partial_t+U\partial_x\big)\phi]& \text { in } \Omega\times (0,T),\\
u(0)=u_0;~~u_t(0)=u_1,\\
u=\Dn u = 0 & \text{ on } \Gamma\times (0,T),\\
(\partial_t+U\partial_x)^2\phi=\Delta \phi & \text { in } \realsthree_+ \times (0,T),\\
\phi(0)=\phi_0;~~\phi_t(0)=\phi_1,\\
\partial_z \phi = \big[(\partial_t+U\partial_x)u \big]_{\text{ext}} & \text{ on } \realstwo_{(x,y)} \times (0,T).
\end{cases}
\end{equation}
Note that, when $\beta=0$ this system collapses to \eqref{flowplate}, the principal model under consideration here.  In this case, we adapt the plate energy $E_{pl}$ and augment it by the quantity $\beta ||u||^2$:
$$E_{pl,\beta}(u)=\dfrac{1}{2}\left[||u_t||^2+\beta ||u||^2+||\Delta u||^2\right]+\Pi(u) = E_{pl}(u) +\frac{\beta}{2} ||u||^2 .$$ In fact, the primary results in \cite{conequil1} consider {\em large} static and viscous damping for the von Karman plate. The main aim of the present paper is to  obtain strong stability result {\it without} any static  damping, i.e. with $\beta =0 $. This will be achieved by asserting  that sufficient amount of stability can  be harvested from the flow alone.  However,  rigorous proof of this  physically plausible fact  is quiet involved and delicate. 
\subsection{Well-posedness and Fundamental Notions}
The dynamics discussed above, when cast in the appropriate framework, are well-posed \cite{b-c,LBC96,springer,jadea12,supersonic,webster}. (For precise definitions of {\em strong, mild, and weak} solutions for the flow-plate dynamics consult \cite{springer,jadea12,conequil1,webster}.)  The following specific well-posedness result is established in \cite{webster,jadea12}:
\begin{theorem}[{\bf Nonlinear Semigroup}]
\label{nonlinearsolution} Assume $U<1$, $p_0 \in L_2(\Omega)$.  Then for any $~T>0$  \eqref{flowplate} has a unique strong (resp. generalized---and hence weak) solution on $[0,T]$, denoted by $S_t (y_0) $, for initial data $y_0=(u_0,u_1;\phi_0,\phi_1) \in Y$. (In the case of strong solutions, the natural compatibility condition must be in force on the data

$ \partial_z \phi_0 = (u_1+Uu_{0x})_{\text{ext}}$.)
Moreover,  $(S_t, Y) $ is a (nonlinear) dynamical
system. Additionally, for $y_0 \in Y_s$, $(S_t,Y_s)$ is also a dynamical system.
 Weak (and hence generalized and strong) solutions satisfy  the following energy  {\it equality}.
 \begin{equation}\label{eident}\ds {\mathcal{E}}(t)+k\int_s^t \int_{\Omega} |u_t(\tau)|^2 d\Omega d\tau= {\mathcal{E}}(s)\end{equation} for $t>s$.
Moreover, the solution $S_t(y_0)$ is stable in the norm of $Y$, i.e., there exists a constant $C(||y_0||_{Y})$ such that for all $ t \geq 0 $ we have: 
\begin{equation}\label{stableS}
 \|S_t (y_0)\|_Y \leq C \left(\|y_0\|_{Y}\right).
 \end{equation}
In addition, the  semigroup $S_t$ is locally Lipschitz on $Y$
\begin{equation}\label{lip}
||S_t(y_1) - S_t(y_2) ||_Y \leq C (R,T) ||y_1-y_2||_Y,~~ \forall ||y_i||_Y \leq R,~~ t \leq T 
\end{equation}
\end{theorem}
\begin{remark}
The proof of Theorem \ref{nonlinearsolution} is based on a suitable  rescaled variant of monotone operator theory   and is given in \cite{webster}. See also \cite{jadea12} for a different approach based on viscosity method. In the case when $U > 1$ the situation is more delicate and dealt with in \cite{supersonic}. In this latter case  (\ref{eident}) no longer holds and 
the inequality in (\ref{stableS}) is valid only locally, for $ t\leq T$ . 
This is  the result of the loss of dissipativity occurring in the supersonic case.
\end{remark}

For the above semigroup we introduce the dynamics operator $\bT:\mathscr{D}(\bT)\subset Y_s \to Y_s$. For its precise structure, we give reference to \cite{jadea12,supersonic,webster} which provide the details of the abstract model.\footnote{We suffice to say that Ball's method provides the generator of the nonlinear semigroup with appropriate dense domain $\cD(\bT)$.} A key property needed in this treatment is: \begin{equation}\label{domainprop}\cD(\bT) \subset  (H^4\cap H_0^2)(\Omega)\times H_0^2(\Omega) \times H^2(\realsthree_+)\times H^1(\realsthree_+). \end{equation}
\begin{remark}
As per \cite{springer,jadea12,webster} the natural {\em invariance} of the dynamics is with respect to the norm $||\cdot||_{Y}$. However, via the hyperbolic-type estimate 
\begin{equation}\label{oneusing*}||\phi(t)||_{L_2(\realsthree_+)} \le ||\phi_0||_{L_2(\realsthree_+)}+\int_0^t||\phi_t(\tau)||_{L_2(\realsthree_+)} d\tau,\end{equation} invariance in $Y_s$ can be recovered {\em on finite time intervals}.
\end{remark}

In order to describe the dynamics of the flow in the context of long-time behavior it is necessary to introduce local spaces, denoted by $Y_{fl,\rho}$: $$\|(\phi_0,\phi_1)\|_{Y_{fl},\rho}\equiv \int_{ K_{\rho} } |\nabla \phi_0|^2  + |\phi_1|^2 d\xb,$$  
where $K_{\rho} \equiv \{ \xb \in \mathbb R^3_{+}; |\xb |\leq \rho \} $. We denote by $Y_{\rho}$ the space $Y_{pl}\times Y_{fl,\rho}$, and we will consider convergence (in time) in $Y_{\rho}$ for any $\rho>0$. Such convergence will be denoted by {\em convergence in the sense of} $\widetilde Y$. 

By virtue of the Hardy inequality \cite[p.301]{springer}
$$\|\phi_0\|^2_{L_2(K_{\rho})} \le C_{\rho}\|\nabla \phi_0\|^2_{L_2(\realsthree_+)}$$ and hence
$$\|(\phi_0,\phi_1)\|_{Y_{fl},\rho}^2 \le  \|(\phi_0,\phi_1)\|_{H^1(K_{\rho})\times L_2(K_{\rho})}^2\le \|(\phi_0,\phi_1)\|_{Y_{fl}}^2.$$

\begin{remark}\label{flownorm} We clarify the relation between the various flow topologies: $H^1(\realsthree_+) \times L_2(\realsthree_+)$, $Y_{fl}$, and $Y_{fl,\rho}$. Initial flow data $(\phi_0,\phi_1)$ are typically chosen in $H^1(\realsthree_+) \times L_2(\realsthree_+)$. Solutions $(u(t),u_t(t);\phi(t),\phi_t(t))$ are global-in-time bounded in the topology of $Y$, but not necessarily global-in-time bounded in the full $Y_s$ norm (owing to the contribution of the $L_2$ norm of $\phi(t)$). Hence, for considerations involving $t \to \infty$, we require a sptially-localized flow perspective. By restricting to any ball $K_{\rho} \subset \realsthree_+$ for the flow dynamics we obtain global-in-time boundedness in the topology of $Y$ as well as local compactness results.  \end{remark}

By the boundedness in Theorem \ref{nonlinearsolution}, the topology corresponding to $Y_{\rho}$ (for any $\rho>0$) becomes a viable measure of long-time behavior which we will refer to as the {\em local energy sense}, and it is appropriate to take limits of the form ~$\displaystyle \lim_{t \to \infty} \|S_t(y_0)\|_{Y_\rho},~~$ for any $\rho>0$, where 
$\displaystyle \|\cdot \|_{Y_\rho} \equiv \|\cdot \|_{Y_{pl} \times Y_{fl,\rho}}$, and $S_t$ is the flow on $Y$ associated to the well-posedness in Theorem \ref{nonlinearsolution} above.

We now note the boundedness (from below) of the nonlinear energy.  Such a bound is necessary to obtain the boundedness of the semigroup quoted in Theorem \ref{nonlinearsolution} above. 
First, we have \cite[Lemma 5.2, p. 3136]{webster}:
\begin{lemma}\label{energybound} Let the hypotheses of Theorem \ref{nonlinearsolution} be in force. 
Then for generalized solutions to \eqref{flowplate}, there exist positive constants $c,C,$ and $M$ positive such that
\begin{equation}  c \mathcal E_*(t)-M_{p_0,b} \le \cE(t) \le C \mathcal E_*(t)+M_{p_0,b}  \end{equation}
\end{lemma}

The proof of Lemma \ref{energybound} given in \cite{webster}  relies on two estimates contraolling lower frequencies. Since these estimates will be used frequently in the sequel, we shall formulate the relevant results. The first estimate controls interactive energy $E_{int}$ on the strength  of   Hardy inequality (see \cite[p. 301]{springer}),
\begin{lemma}\label{EIN}
 \begin{equation} |E_{int}(t)| \le U\delta \|\nabla \phi(t)\|_{\realsthree_+}^2+\frac{C\cdot U}{\delta}\|u_x(t)\|_{\Omega}^2, ~~\delta>0, \end{equation}
 \end{lemma}
The second one controls lower frequencies  \cite[p. 49]{springer} by exploiting nonlinear effects of potential energy:
\begin{lemma}\label{l:epsilon}
For any $u \in H^2(\Omega) \cap H_0^1(\Omega) $ and   $\epsilon > 0 $ there exists $M_{\epsilon} $ such that
$$\|u\|^2_{\Omega} \leq \epsilon [\|\Delta u\|^2_{\Omega}  + \Pi_*(u) ] + M_{\epsilon}.$$
\end{lemma}
\begin{remark}
In fact, in the case of Berger's nonlinearity Lemma \ref{l:epsilon} holds without the term $||\Delta u||^2_{\Omega} $. However, in the case of other nonlinearities, such as von Karman, the presence of $\Delta$ is indispensable. For this reason we cite a more general version of this lemma which can be used for von Karman models as well. \end{remark}

 {\em Note: global-in-time boundedness of solutions cannot be obtained without accounting for nonlinear effects.}
From the above lemmata  and energy inequality  we also  have (see \cite{jadea12,webster}):
\begin{lemma}\label{globalbound} Let the hypotheses of Theorem \ref{nonlinearsolution} be in force. Then any generalized (and hence weak) solution to \eqref{flowplate} satisfies the bound \begin{equation}\label{apr}
\sup_{t \ge 0} \left\{\|u_t\|_{\Omega}^2+\|\Delta u\|_{\Omega}^2+\|\phi_t\|_{\realsthree_+}^2+\|\nabla \phi\|_{\realsthree_+}^2 \right\}  \leq C\big(\|y_0\|_Y\big)< + \infty.\end{equation}
\end{lemma}
\noindent Moreover, from energy inequality \ref{eident} and  (\ref{apr}) the following dissipation integral is finite (see \cite{springer,conequil1}):
\begin{corollary}\label{dissint} Let the hypotheses of Theorem \ref{nonlinearsolution} be in force and assume
$k> 0$; then the dissipation integral is finite. Namely, for a generalized solution to \eqref{flowplate}, we have $$  \int_0^{\infty} \|u_t(t)\|_{L_2(\Omega)}^2 dt \leq  K < \infty,$$
where $K$ depends on the particular trajectory.
\end{corollary}
We now state and discuss the stationary problem associated to \eqref{flowplate}, which has the form:
\begin{equation}\label{static}
\begin{cases}
\Delta^2u+f(u)=p_0(\xb)+Ur_{\Omega}tr[\partial_x \phi]& \xb \in \Omega\\
u=\Dn u= 0 & \xb \in \Gamma\\
\Delta \phi -U^2 \partial_x^2\phi=0 & \xb \in \realsthree_+\\
\Dz \phi = U\partial_x u_{\text{ext}}  & \xb \in \partial \realsthree_+
\end{cases}
\end{equation}
This problem has been studied in long-time behavior considerations for flow-plate interactions, most recently in \cite[Section 6.5.6]{springer}; in this reference, the following theorem is shown for subsonic flows (this is given as \cite[Theorem 6.5.10]{springer}):
\begin{theorem}\label{statictheorem}
Suppose $0 \le U <1$, $ k \geq 0$, with $p_0 \in L_2(\Omega)$. Then {\em weak} solutions $\left(u(\xb),\phi(\xb)\right)$ to \eqref{static} exist and satisfy the additional regularity property $$(u,\phi) \in (H^4\cap H_0^2)(\Omega) \times W^2(\realsthree_+).$$
Moreover, the stationary solutions mentioned above correspond to the extreme points of the  potential energy functional $$D(u,\phi) = \frac{1}{2}\|\Delta u\|_{\Omega}^2+\Pi(u)+\frac{1}{2}\|\nabla \phi\|_{\realsthree_+}^2-\dfrac{U^2}{2}\|\partial_x\phi\|_{\realsthree_+}^2 + U\lb\partial_x u, tr[\phi]\rb_{\Omega},$$ considered for $(u,\phi) \in H_0^2(\Omega) \times W^1(\realsthree_+)$. \end{theorem} \noindent
The potential energy $D(u, \phi) $ is smooth on the space  $H_0^2(\Omega) \times W^1(\mathbb R^3_{+})$ and moreover
$$D(u,\phi) \geq c \left[\|\Delta  u\|^2_{\Omega}  + \|\nabla \phi\|_{\realsthree_+}^2\right] - C.$$
This latter property is a consequence of Lemma \ref{l:epsilon}.
Thus,  it achieves its minimum, and the extremal set of the functional $D$ is non-empty.
 {\em We denote the set of all stationary solutions (weak solutions to \eqref{static} above) as $\mathcal N$,} that is
 $$\mathcal N \equiv \{(\hat u,\hat \phi) \in H_0^2(\Omega) \times W^1(\realsthree_+): (\hat u,\hat \phi) ~\text{satisfy \eqref{static} variationally}\}.$$

	\subsection{Statement of Main Result}\label{mainresults} 
		\begin{theorem}\label{conequil}[Principal Result]
		
Let $0\le U<1$ and assume $p_0 \in L_2(\Omega)$ and $b\ge 0$. Assume $y_0=(u_0,u_1;\phi_0,\phi_1) \in Y$. Then there is a minimal damping coefficient $k_{min}>0$ (not depending on the particular solution) so that for $k\ge k_{min}>0$ any generalized solution $(u(t),\phi(t))$ to the system with localized initial flow data (i.e., the supports of $\phi_0$ and $\phi_1$ are contained in some $K_{\rho_0}$) has the property that \begin{align*}\lim_{t \to \infty} \inf_{(\hat u,\hat \phi) \in \mathcal N}\left\{\|u(t)-\hat u\|^2_{H^2(\Omega)}+\|u_t(t)\|^2_{L_2(\Omega)}+\|\phi(t)-\hat \phi\|_{H^1( K_{\rho} )}^2+\|\phi_t(t)\|^2_{L_2( K_{\rho} )} \right\}\\=0, \text{ for any }  \rho>0.\end{align*} \end{theorem}
\begin{remark}\label{kmin} In this result, and all results below, the minimal damping coefficient $k_{min}$ depends on the loading $p_0$ and $b$, as well as the domain $\Omega$, the flow support parameter $\rho_0$, and the unperturbed flow velocity $U$, but are independent on the particular initial data of the system. \end{remark}

If we make a further physical assumption that  $\mathcal N$ is an isolated set (e.g., finite), we have the following second main result as a corollary:
\begin{corollary}\label{improve} Assume that $\mathcal N$ is an isolated set.
Let the hypotheses of Theorem \ref{conequil} be in force; then for any generalized solution $(u,\phi)$ to \eqref{flowplate} (with localized flow data, as above), there exists a solution $(\hat u, \hat \phi)$ satisfying \eqref{static} such that
\begin{align*}\lim_{t \to \infty} \left\{\|u(t)-\hat u\|^2_{H^2(\Omega)}+\|u_t(t)\|^2_{L_2(\Omega)}+\|\phi(t)-\hat \phi\|_{H^1( K_{\rho} )}^2+\|\phi_t(t)\|^2_{L_2( K_{\rho} )} \right\}\\=0, \text{ for any }  \rho>0.\end{align*}
\end{corollary}

\begin{remark}\label{sard}
 For given loads $b$ and $p_0$, the set of solutions is generically finite. This is to say that there is an open dense set $\mathcal R \subset L_2(\Omega)\times \mathbb R$ such that  if $(p_0,b) \in \mathcal R$ then the corresponding set of stationary solutions $\mathcal N$ is finite. This follows from the Sard-Smale theorem, as shown in \cite[Theorem 1.5.7 and Remark 6.5.11]{springer}.
\end{remark}

In addition to the primary results, we mention some results obtained along the way which are of interest in their own right. Owing to the necessary background information to precisely state these results, at this stage we describe them only informally. 

First, in Section \ref{highernorms00} (following a few technical descriptions and lemmas) we prove the main supporting result: Theorem \ref{highernorms0}. This result says that global-in-time bounds on plate dynamics in higher norms yields the desired convergence of the flow-plate dynamics at the finite-energy level (in $Y_{\rho}$, for all $\rho>0$). It appeared in \cite{conequil1} as a step in the proof of the main result, but was not presented independently. 

In Section \ref{propagate} we state and prove Theorem \ref{highernorms}. This result shows that for ``smooth" initial data for the full flow-plate dynamics we can propagate regularity on the finite time horizon. This result (or slightly modified versions of it) will be used many times below. For the plate dynamics alone

this regularity can be propagated on the infinite time horizon (with corresponding uniform bound). Moreover, this result {\em does not depend on the size of the damping}.

\section{Technical Preliminaries}\label{techpres}
In this section we discuss a few key technical results and constructions which will be used critically later. References are provided for each of the considerations below.
\subsection{Flow Potentials with Given Neumann Data}
In what follows it will be necessary to consider the flow equation with {\em given} Neumann data. Consider the problem:
\begin{equation}\label{floweq*}
\begin{cases}
(\partial_t+U\partial_x)^2\phi=\Delta \phi & \text{ in }~\mathbb R_+^3\\
\partial_{\nu} \phi\Big|_{z=0} = h(\xb,t) & \text{ in }~\mathbb R^2\\
\phi(t_0)=\phi_0;~~\phi_t(t_0)=\phi_1
\end{cases}
\end{equation}
We have the following theorem from \cite{b-c-1,springer,miyatake1973}:
\begin{theorem}\label{flowpot}
Assume $U\ge 0$, $U\ne 1$; take $(\phi_0,\phi_1) \in H^1(\realsthree)\times L_2(\realsthree).$ If \newline$h \in C\left([t_0,\infty);H^{1/2}(\mathbb R^2)\right)$ then \eqref{floweq*} is well-posed (in the weak sense) with 
$$\phi \in C\left([t_0,\infty);H^1(\realsthree_+)\right),~~\phi_t \in C\left( [t_0,\infty);L_2(\realsthree_+)\right).$$
\end{theorem}
\begin{remark}
In fact, a stronger regularity result is available. Finite energy ($H^1(\Omega) \times L_2(\Omega)$) solutions are obtained with $h\in H^{1/3}((0,T) \times \mathbb R^2)$ \cite{LT,tataru}. However, the corresponding estimate does not have sufficiently  good control of  the dependence on time $T$, and hence is of limited applicability in the present (stabilization) context. 
\end{remark}

\subsection{Flow Superposition and Compactness Criteria}
We may decompose the flow problem from \eqref{floweq*} into two pieces corresponding to zero Neumann data, and zero initial data, respectively:
\begin{equation}\label{flow1}
\begin{cases}
(\partial_t+U\partial_x)^2\phi^*=\Delta \phi^*  & ~\text{ in }~ \realsthree_+ \times (0,T) \\
\phi^*(0)=\phi_0; ~~\phi_t(0)=\phi_1 \\
\Dn \phi^* = 0 & ~\text { in } ~\partial \realsthree_+ \times (0,T)
\end{cases}
\end{equation}
\begin{equation}\label{flow2}
\begin{cases}
(\partial_t+U\partial_x)^2\phi^{**}=\Delta \phi^{**}  & ~\text{ in }~ \realsthree_+ \times (0,T) \\
\phi^{**}(0)=0; ~~\phi_t^{**}(0)=0 \\
\Dn \phi^{**} = h(\xb,t) & ~\text { in } ~\partial \realsthree_+ \times (0,T)
\end{cases}
\end{equation}
where we will later take:
\begin{equation}\label{h}
h(\xb, t) \equiv -[u_t + U u_x]_{\text{ext}} \in C(L_2(\mathbb R^2))
\end{equation}
Following the analyses in \cite{b-c,b-c-1,springer,ryz,ryz2} we denote the solution to \eqref{flow1} as $\phi^*$ and the solution to \eqref{flow2} as $\phi^{**}$; then, the full flow solution $\phi$ coming from \eqref{floweq*} has the form $$\phi(t)=\phi^*(t)+\phi^{**}(t).$$ 
\begin{remark}
The analysis of $\phi^*$ below is identical to that given in \cite{b-c,b-c-1,springer,ryz}. However, the treatment of the second part
$\phi^{**} (t)$, which corresponds to the  hyperbolic Neumann map,  is very different than these references due to the loss of a derivative in the Neumann map. Indeed, with rotational inertia in place ($u_t \in H^1(\Omega)$) one has (for finite energy solutions)  $h \in C(0, T; H^1(\Omega) ) $.
On the other hand from  \cite{miyatake1973} $ h \in L_2(H^{1/2}(\realstwo) ) \mapsto
\phi^{**} \in  C(H^1(\realsthree_+) \cap C^1(L_2(\realsthree_+)) $---where the latter  is of  finite energy. Thus the Neumann map is {\em compact} in this case.
In the absence of rotational inertia one has only that  $h \in C(L_2(\realstwo)) $. Such boundary regularity does not reproduce finite energy
solutions,  with a maximal regularity  $ \phi^{**} \in  C(H^{2/3}(\realsthree_+)) \cap C^1(H^{-1/3}(\realsthree_+)),$ yielding  the loss of $1/3 $ derivative. This loss is optimal and cannot be improved  \cite{tataru}. This fact underscores that the entirely component-wise analysis of finite energy solutions for the flow-plate model successfully performed in the  past literature
\cite{springer,ryz,ryz2} cannot be utilized here.
  \end{remark}
 For the analysis of $\phi^*$ we use the tools developed in \cite{b-c,b-c-1}.
Using the Kirchhoff type representation for the solution  $\phi^*(\xb,t)$
in $\R_+^3$ (see, e.g., Theorem~6.6.12 in \cite{springer}), we can conclude that
if the initial data   $\phi_0$ and $\phi_1$ are  localized in the ball $ K_{\rho}$,
then by  finite dependence on the domain of the signal in three dimensions  (Huygen's principle),
   one obtains  $\phi^*(\xb,t)\equiv 0$ for all $\xb\in  K_{\rho} $
and $t\ge t_{\rho}$. Thus we have that
\[
\big(\partial_t+U\partial_x\big)tr[\phi^*]\equiv0,~~~\xb\in \Omega,~t\ge t_{\rho}.
\]
Thus $\phi^*$ tends to zero in the sense of the local flow energy, i.e.,  \begin{equation}\label{starstable} \|\nabla \phi^*(t)\|_{L_2( K_{\rho} )}^2 + \|\phi^*_t(t)\|_{L_2( K_{\rho} )} \to 0, ~~ t \to \infty,\end{equation} for all fixed $ \rho>0$.

We now introduce a compactness criterion in the {\em local flow energy sense}:
\begin{lemma}\label{compactnesscriterion}
Let $\{(\phi_0^m,\phi_1^m)\}_m^{\infty}$ be a bounded sequence in $Y_{fl}=W_1(\realsthree_+) \times L_2(\realsthree_+)$ and let $\eta>0$. If for any $ \rho>0$ there exists an $N(\rho)$ and $C(\rho)$ so that
$$\|\nabla \phi_0^m\|^2_{\eta, K_{\rho} }+\|\phi_1^m\|^2_{\eta, K_{\rho} } \le C(\rho)$$ for all $m>N(\rho)$ then the sequence $\{(\phi_0^m,\phi_1^m)\}_m^{\infty}$ is compact in $\widetilde Y_{fl}$.
\end{lemma}
\noindent This is given as Lemma 10 in \cite[p. 472]{ryz} (where it is proved) and is utilized in a critical way in \cite{ryz2} as well.

For the term $\phi^{**}$ we have the following theorem which provides us with an explicit form of the solution (for a proof, see for instance \cite[Theorem 6.6.10]{springer}).
Below, we utilize the notation: $$f^{\dag}(\xb,t,s,\theta)=f\left(x-\kappa_1(\theta,s,z), y-\kappa_2(\theta,s,z), t-s\right),$$ and $$\kappa_1(\theta,s,z)=Us+\sqrt{s^2-z^2}\sin \theta,~~\kappa_2(\theta,s,z) = \sqrt{s^2-z^2}\cos \theta.$$
\begin{theorem}\label{flowformula}
Considering the problem in \eqref{flow2} with  zero initial flow data, and considering $$\ds h(\xb,t) =-[u_t (x,y,t)+Uu_x(x,y,t)]_{\text{ext}},$$ there exists a time $t^*(\Omega,U)$ such that we have the following representation for the weak solution $\phi^{**}(t)$ for $t>t^{*}$:
\begin{equation}
\phi^{**}(\xb,t) = -\dfrac{\chi(t-z) }{2\pi}\int_z^{t^*}\int_0^{2\pi}(u^{\dag}_t(\xb,t,s,\theta)+Uu^{\dag}_x(\xb,t,s,\theta))d\theta ds.
\end{equation}
where $\chi(s) $ is the Heaviside function. The time $t^*$ is given by:
\begin{equation*}
   t^*=\inf \{ t~:~\xb(U,\theta, s) \notin \Omega \text{ for all } (x,y) \in \Omega, ~\theta \in [0,2\pi], \text{ and } s>t\}
\end{equation*} with $\xb(U,\theta,s) = (x-(U+\sin \theta)s,y-s\cos\theta) \subset \realstwo$ (not to be confused with $\xb$ having no dependencies noted, which is simply $\xb = (x,y)$).

Moreover, we have the following point-wise formula for the derivative in $t$ \cite[p. 480]{ryz2}(which is justified for plate solutions with initial data in $\cD(\bT)$, and can be taken distributionally for data in $Y$):
\begin{align}\label{formderiv1}
\phi^{**}_t&(\xb, t) =  \dfrac{1}{2\pi}\Big\{\int_0^{2\pi}u_t^{\dag}(\xb,t,t^*,\theta) d\theta -\int_0^{2\pi} u_t^{\dag}(\xb,t,z,\theta )d\theta \\\nonumber
&+U\int_z^{t^*}\int_0^{2\pi}[\partial_x u_t^{\dag}](\xb,t,s,\theta) d\theta ds+\int_z^{t^*}\int_0^{2\pi}\dfrac{s}{\sqrt{s^2-z^2}}[M_{\theta}u_t^{\dag}](\xb,t,s,\theta) d\theta ds\Big\} 
\end{align}
with $M_{\theta} = \sin\theta\partial_x+\cos \theta \partial_y$.
\end{theorem}

We note that  with  $(\phi_0,\phi_1) \in H^1(\R_+^3)\times  L_2(\R_+^3) $
one obtains \cite{supersonic,miyatake1973}\newline
 $(\phi^*(t),  \phi^{*}_t(t)) \in   H^1(\R_+^3)\times  L_2(\R_+^3)$. Thus, by Theorem \ref{nonlinearsolution}
 we also have that
 $$
 (\phi^{**}(t),  \phi^{**}_t(t)) \in  H^1(\R_+^3)\times  L_2(\R_+^3).
 $$
 \begin{remark} Note that \textit{this last property is not valid for a flow solution with $h \in L_2(\mathbb R^2)$ boundary Neumann data.} The general theory will provide at most \newline $ H^{2/3}(\R_+^3\times [0,T])$ \cite{tataru}.  However, the improved regularity is due to the interaction with the plate and the resulting cancellations on  the interface. Moreover, we also obtain a meaningful ``hidden trace regularity" for the aeroelastic potential on the boundary of the structure \cite{fereisel,supersonic}:
  \begin{equation}\label{trace}
  (  \partial _{t} + U \partial_{x} )tr [ \phi ] \in L_2(0, T; H^{-1/2}(\mathbb R^2) )
  \end{equation}
  where $T$ is arbitrary. \end{remark}

Additionally, the following inequality has been shown in \cite{ryz}, labeled Lemma 8 (p. 469) and (56) (p. 479):
\begin{lemma}\label{compact2}
For \eqref{flow2} taken with $h(\xb,t)=-(u_t+Uu_x)_{\text{ext}}$, we have
\begin{align}
\|\nabla  \phi^{**}(t)\|^2_{\eta, K_{\rho} }&+\|\phi_t^{**}(t)\|^2_{\eta, K_{\rho} }\nonumber \\
\le  ~C(\rho)& \big\{\|u(t)\|^2_{H^{s+\eta}(t-t^*,t;H_0^{2+\eta}(\Omega))}+\|u_t(t)\|^2_{H^{s+\eta}(t-t^*,t;H_0^{1+\eta}(\Omega))} \big\}
\end{align} for $0<s+\eta<1/2$ and $t>t^*(U,\Omega)$.
\end{lemma}
\begin{remark}
The above bound was critical in the previous analyses of this problem which made use of inertial terms $-\alpha \Delta u_{tt}$ and strong damping $-k_2\Delta u_t$ \cite{b-c,b-c-1} or thermal smoothing \cite{ryz,ryz2}. 
\end{remark}

	\subsection{Nonlinear Plates with Delay}
	In this section, the full flow-plate problem is reduced into a delayed plate problem. This delayed plate is subject three conflicting actions contributed by flow:
\begin{enumerate}
\item The delay term itself: $q^u(t)$.
\item The term $U u_x$ in the plate equation---nonconservative, potentially leading to chaos.
\item The term $u_t$ (with correct sign) in the plate equation, generated by flow. \end{enumerate}

Above, (3.) is stabilizing. But (1.) and (2.) are ``bad" and do not stabilize the plate. 
It is also known that a delay term can produce destabilizing effects on the dynamics, {\em unless} the damping ~$(k+1) u_t$ is sufficiently large \cite{pignotti1}. 
This is one contributing factor for our need of a large damping---$k$ value. 

On the other hand, the term $u_x$ can typically be dealt with via a large static damping \cite{conequil1}.  Here, we do not assume any static damping for the overall model . Our challenge, then, in this paper  is to show that this potentially chaotic term is ``thrown" to the equilibria set, and that the dynamics converge to it.
	
The fundamental work in \cite{delay} focuses on the reduction of the model in \eqref{flowplate} to a delayed plate equation:
\begin{theorem}\label{rewrite}
Let the hypotheses of Theorem~\ref{nonlinearsolution} be in force,
and $(u_0,u_1;\phi_0,\phi_1) \in H_0^2(\Omega) \times L_2(\Omega) \times H^1(\realsthree_+) \times L_2(\realsthree_+)$.
Assume that there exists an $\rho_0$ such that $\phi_0(\xb) = \phi_1(\xb)=0$ for $|\xb|>\rho_0$.
Then the there exists a time $t^{\#}(\rho_0,U,\Omega) > 0$ such that for all $t>t^{\#}$ the weak solution $u(t)$ in
(\ref{flowplate})  satisfies the following equation:
\begin{equation}\label{reducedplate}
u_{tt}+\Delta^2u+ku_t+f(u)=p_0-(\partial_t+U\partial_x)u-q^u(t)
\end{equation}
with
\begin{equation}\label{potential}
q^u(t)=\dfrac{1}{2\pi}\int_0^{t^*}ds\int_0^{2\pi}d\theta [M^2_{\theta} u_{\text{ext}}](x-(U+\sin \theta)s,y-s\cos \theta, t-s).
\end{equation}
\end{theorem}

Owing to this reduction we can choose to study a general delayed plate model (as was done in \cite{delay}) to produce maximal generality with respect to our results for the delay system (which can be viewed as independent of our results for the full flow-plate dynamics): 

\begin{equation}\label{gendelay}
\begin{cases} u_{tt}+\Delta^2u+k_0u_t +f(u) = p_0-q(u^t)+Lu~~~~\text{in} ~~\Omega\\
u=\partial_{\nu}u=0~~~~\text{on}~~\Gamma\\
u(0)=u_0,~u_t(0)=u_1,~u|_{(-t^*,0)}=\eta\end{cases}
\end{equation}
The linear operator $L:H_0^2(\Omega) \to H^{2-\delta}(\Omega)$ corresponds to non-dissipative lower-order terms : $ Lu = - U u_x $  and $ k_0 = k +1 $. We have utilized the standard notation that $u^t(s)=u(t+s)$, so $u^t$ denotes $u(s)$ for all $s\in(t-t^*,t)$. The operator $q(u^t)$ represents a general delay type operator (depending on $u^t$), of which $q^u(t)$ is a specific example. (This notation is consistent with \cite{delay}.)

We will also need the following formula for the time derivative of $q^u(t)$ appearing above in \eqref{potential}:
\begin{align}\label{qderiv}
\partial_t [q^{u}](t) =&\int_0^{2\pi}\frac{1}{2 \pi}[M^2_{\theta} u]_{\text{ext}}\big(\xb(U,\theta,0),t \big)d\theta\\
&-\int_0^{2\pi}\frac{1}{2 \pi}[M^2_{\theta} u]_{\text{ext}}\big(\xb(U,\theta,t^*),t-t^* \big)d\theta\nonumber\\
&+\Big( \int_0^{t^*}\int_0^{2\pi}(U+\sin\theta)\frac{1}{2 \pi}[M^2_{\theta} u_x]_{\text{ext}}\big(\xb(U,\theta,s),t-s \big)d\theta ds\Big)\nonumber\\
&+\Big(\int_0^{t^*}\int_0^{2\pi}(\cos\theta)\frac{1}{2 \pi}[M^2_{\theta} u_y]_{\text{ext}}\big(\xb(U,\theta,s),t-s \big)d\theta ds\Big). \nonumber
\end{align}

The proofs of Theorems \ref{th:main2} and \ref{indeppara} cited below are based on an analysis of the delay evolution $(T_t,\mathbf H)$, with $\mathbf H \equiv H_0^2(\Omega)\times L_2(\Omega)\times L_2\left(-t^*,0;H_0^2(\Omega)\right),$ corresponding to the reduced plate with given data $x_0 \in \mathbf H$ we have that $T_t(x_0)=\left(u(t),u_t(t);u^t\right)$ with $x_0=(u_0,u_1,\eta)$. The norm is taken to be $$||(u,v;\eta)||^2_{\mathbf H} \equiv ||\Delta u||^2+||v||^2+\int_{-t^*}^0||\Delta \eta(t+s)||^2 ds.$$
In application, we will consider an initial datum $y_0 \in Y$ corresponding to the dynamics in \eqref{flowplate} (the full flow-plate dynamics) at time $t=0$. We employ the reduction result  Theorem \ref{rewrite}, and we may consider the ``initial time" ($t=t_0$) for the delay dynamics corresponding to any time after the reduction time $t^{\#}(\rho_0,U,\Omega)$ above. At such a time, the data which is fed into \eqref{reducedplate} is $x_0=(u(t_0),u_t(t_0),u^{t_0})$, where this data is determined by the full dynamics of \eqref{flowplate} on $(t_0-t^*,t_0)$. (That \eqref{flowplate} is well-posed on $Y$ and \eqref{gendelay} is well-posed on $\mathbf H$ (see \cite{delay}) allows us to move back and forth between the two descriptions of the system.)

\subsection{Attractors}\label{attractors}
A {\em global attractor} $\mathcal A$ for a dynamical system $(T_t,\mathbf H)$ is an invariant, compact set which is uniformly attracting, in the sense that if $B \subset \mathbf H$ is a bounded set, we have that $$\lim_{t \to \infty} d_{\mathbf H}(T_t(B),\mathcal A) =0,$$ (where we have utilized the Hausdorff semi-distance in $\mathbf H$).

	\section{Supporting Results}\label{prevlit}
	For self-containedness of exposition we now provide a summary of the key results of \cite{delay} which will be used critically in this treatment.

		\subsection{Plate Dynamics}
	The first theorem we present is the main result of \cite{delay}, given there as Theorem 3.4.
	 \begin{theorem}\label{th:main2} Consider any nonlinearity of physical type. 
Suppose $0\le U \ne 1$ and $p_0 \in L_2(\Omega)$. Take any $k \ge 0$
 Then there exists a compact set $\mathcal U \subset H_0^2(\Omega) \times L_2(\Omega)$ of finite fractal dimension such that $$\lim_{t\to\infty} d_{Y_{pl}} \Big( (u(t),u_t(t)),\mathcal U\Big)=\lim_{t \to \infty}\inf_{(\nu_0,\nu_1) \in \mathcal U} \big( \|u(t)-\nu_0\|_2^2+\|u_t(t)-\nu_1\|^2\big)=0$$
for any weak solution $(u,u_t;\phi,\phi_t)$ to (\ref{flowplate}), with
initial data
$
(u_0, u_1;\phi_0,\phi_1) \in  Y_{\rho_0},$ i.e., which are
localized  in $\R_+^3$---$\phi_0(\xb)=\phi_1(\xb)=0$ for $|\xb|>\rho_0$ for some $ \rho_0>0$. We have the extra regularity $\mathcal U \subset \big(H^4(\Omega)\cap H_0^2(\Omega) \big) \times H^2(\Omega)$, and any plate trajectory $(u,u_t)$ on the plate attracting set (namely with initial $(u_0,u_1) \in \mathcal U$) has the additional property that $$(u,u_{t},u_{tt}) \in C\left(\mathbb R;(H^4\cap H_0^2)(\Omega) \times H_0^2(\Omega) \times L_2(\Omega)\right).$$
\end{theorem}

In the setting of \cite{conequil1} and the present analysis we are interested in obtaining a sufficiently large damping so as to guarantee strong convergence of full flow-plate trajectories. It is imperative here that there exists a uniform absorbing ball for the plate dynamics which is independent of the damping parameter (when it is sufficiently large). The following result is given Lemma 4.6 in \cite{conequil1} (there applying to \eqref{withbeta} for $\beta\ge0$).
\begin{lemma}\label{indeppara}
With the same hypotheses as Theorem \ref{th:main2} above, there is an absorbing set $\mathcal B \subset H_0^2(\Omega)\times L_2(\Omega)$ which is (i) uniform with respect to the initial data $y_0$ (though it may depend on $\rho_0$)  and (ii) uniform in the damping parameter $k \ge k_*> 0$, for some $k_*$. \end{lemma}

\begin{remark}\label{beta}
When we include the effects of static damping---of the form ~$\beta u$, as in \eqref{withbeta}---in the plate dynamics, we also have the existence of the attractor (and corresponding absorbing set) whose size also can be made independent of $\beta \ge 0$.
\end{remark}

 The results in \cite{delay} perform analysis on the abstract delay equation \eqref{gendelay} on $\mathbf H$, including energy methods which show the existence of a true global attractor $\mathscr U$ (as well as an absorbing set $\mathscr B \supset \mathscr U$) for the the delay evolution $(T_t,\mathbf H)$. These results depend on taking $k_0>0$ (where $k_0$ is the damping coefficient in \eqref{gendelay}). (Note that, for the specific delayed plate of interest here \eqref{reducedplate}, we have $k_0=k+1$, and thus we could take $k \in (-1,\infty]$.) Then, to achieve the results in Theorems \ref{th:main2} and \ref{indeppara} projections are taken which suppress the delay component of the evolution. 
 
  \begin{remark} In what follows, the notation $\mathscr U\subset \mathscr B$ correspond to the global attractor and (a fixed) uniform absorbing set for the dynamics $(T_t,\mathbf H)$. The notation $\mathcal U  \subset \mathcal B$ will refer to the attracting set and absorbing ball for the {\em projected dynamics} $(u,u_t)$ of $(S_t,Y)$ (delay and flow components dropped).
 \end{remark}
\subsection{Weak Stability}
Given the existence of a compact global attracting set for the plate dynamics, as well as the finiteness of the dissipation integral (Corollary \ref{dissint}), we have the following theorem (given as Theorem 7.1 \cite{conequil1} and applying for \eqref{withbeta} for $\beta \ge 0$):
\begin{theorem}\label{convergenceprops}
Assume $k> 0$. For any initial data $y_0=(u_0,u_1;\phi_0,\phi_1)\in Y$ and any sequence of times $t_n \to \infty$ there is a subsequence $t_{n_j}$ and a point $\widehat y = (\widehat u,\widehat w;\widehat \phi, \widehat \psi)$ so that:
 \begin{enumerate}
 \item $S_{t_{n_j}}(y_0) \rightharpoonup  \widehat y$. \vskip.2cm
\item
$\| u(t_{n_j}) - \widehat{u}\|_{2,\Omega} \rightarrow 0$. \vskip.2cm
\item $\|u_t(t)\|^2_{0} \to 0, ~t \to \infty$, and hence $\widehat w =0$. \vskip.2cm
\item Along the sequence of times $t_{n_j} \to \infty$ we have \begin{align}
\sup_{\tau \in [-c,c]}\|u(t_{n_j}+\tau)-\widehat u\|_{2-\delta,\Omega} \to 0 ~\text{ for any fixed}~ c > 0, ~\delta\in (0,2).
\end{align}
\end{enumerate}
\end{theorem}
In view of the above result, it is clear that the general, novel mathematical  challenge  in resolving the flutter problem is to improve the convergence of the flow component of the dynamics from {\it weak} to {\it strong}. This is a highly non-trivial task,  taking into consideration the lack of compactness properties of the dynamics, along with the lack of dissipation acting on the flow equation. 
\section{From Weak to Strong}\label{wts}
We present (in a reorganized way) the key past results in \cite{ryz,ryz2}, which yield supporting results that can then be applied in the proof of our main results here. Additionally, we restate in the present context certain ideas and results which were critically used in \cite{conequil1}. 
\subsection{Hadamard Continuity}
The  key results below address the {\em continuous dependence} property of the dynamics, namely, that $S_t(\cdot)$ is Hadamard continuous on $Y$ (and $Y_s$ on {\em finite time intervals}). This was first addressed within the proof of Theorem 3.1 \cite[pp. 1963--1964]{conequil1}, and the proof of both Lemma \ref{hadcont} and Theorem \ref{strongcont} can be found in \cite{conequil1}. (In fact, the Hadamard continuity results below are analogously true for von Karman dynamics.)
\begin{lemma}\label{hadcont} Let $0\le U<1$ and assume $p_0 \in L_2(\Omega)$. Assume $y_0=(u_0,u_1;\phi_0,\phi_1) \in Y$.
Consider the dynamics generated by generalized solutions to \eqref{flowplate} (or \eqref{withbeta} if $\beta>0$), denoted $S_t(y_0)$. For any $k\ge 0,~ \beta \ge 0$, and for any fixed $T>0$, we have
 $$S_{t} (y_0^m) \rightarrow y(t)=S_t(y_0)  \in Y,$$ when $y^m_0 \to y_0 \in Y,$ uniformly
 for all $t< T$; this convergence may depend---possibly---on $T > 0$.
\end{lemma}
\begin{remark}
The corresponding delay system $(T_t,\mathbf H)$ exhibits the Hadamard continuity property on any $[t_0,T]$ as well. This follows from Lemma \ref{hadcont} and Theorem \ref{rewrite}. Moreover, in \cite{delay} the {\em continuous dependence} property is shown directly from energy estimates on the delay system \eqref{gendelay}.
\end{remark}
\begin{remark}\label{passing0}
At present, it is not clear that the full dynamics for \eqref{flowplate} are {\em global-in-time} Hadamard continuous. Whether this is the case is related to the decay rate of $||u_t(t)||_{L_2(\Omega)} \to 0$; indeed, to obtain this global-in-time property it is sufficient for $u_t \in L_1(T,\infty;L_2(\Omega))$. With $k>0$,  we do have $u_t \in L_2(0, \infty;L_2(\Omega))$; however, improving $L_2(0,\infty)$ to $L_1(0;\infty)$ in general requires very substantial work. See Remark \ref{passing}.
\end{remark}
If we consider  the addition of {\em large static damping} (see Remark \ref{beta}), as well as large viscous damping, we see that the Hadamard continuity can be improved to {\em global-in-time} (this was one of the primary results in \cite{conequil1}):
\begin{theorem}\label{strongcont}
Let $0\le U<1$ and assume $p_0 \in L_2(\Omega)$. Assume $y_0=(u_0,u_1;\phi_0,\phi_1) \in Y$. 

Consider the dynamics generated by generalized solutions to \eqref{withbeta} (i.e., with $\beta >0$), denoted $S_t(y_0)$. Assuming that both damping parameters are sufficiently large $ k \geq  k_c > 0$, $\beta \ge \beta_c>0$,  the semigroup $S_t(\cdot)$ {\em is} uniform-in-time Hadamard continuous, i.e., for any sequence $y_0^m \to y_0$ in $Y$ and any $\epsilon>0$ there is an $M$ so that for $m>M$
$$\sup_{t>0} \|S_t(y^m_0) - S_t(y_0)\|_{Y_{\rho}} < \epsilon. $$
\end{theorem}
This theorem and its proof are given in Section 7 of \cite{conequil1}. Theorem \ref{strongcont} will be used below in an auxiliary fashion on a decomposed portion of the full flow-plate dynamics which resembles \eqref{withbeta} (i.e., has static damping added)---Theorem \ref{hadcont*}. 

\begin{remark} The minimal damping coefficients $k_c$ and $\beta_c$ for the result above depend on the invariant set $\mathscr B$ for the delayed plate dynamics, which itself depends on the loading $p_0$ and $b$, as well as the domain $\Omega$ and  other physical constants such as  $U,\rho_0$, but {\em is independent on the particular initial data of the system}. \end{remark}

\subsection{Bounds in Higher Norms}\label{highernorms00}
 We now state and prove the critical supporting theorem for this treatment (and also \cite{conequil1}) Theorem \ref{highernorms0}: \begin{quote} For finite energy solutions to the flow-plate system, a uniform-in-time bound on plate solutions in a higher topology yields the desired convergence to equilibria result (as in Theorem \ref{conequil}) in the topology $Y_{\rho}$ for any $\rho>0$.\end{quote} In other words, beginning with finite-energy solutions, if we further know that the plate dynamics are sufficiently regular, then compactness properties can be transferred to the flow dynamics. This result is independent of the particular structure of the plate dynamics---only global-in-time bounds in higher norms of the plate solution $(u,u_t)$ are needed; indeed, Theorem \ref{highernorms0} follows from the structure of the (linear) coupling via the Neumann flow condition, written in terms of $[u_t+Uu_x]_{\text{ext}}$, and appearing in Theorem \ref{flowformula}. 

\begin{theorem}\label{highernorms0} 
Suppose $0\le U \ne 1$ and $p_0 \in L_2(\Omega)$, and take any $k > 0$.  Let $(u,\phi)$ be a weak solution to \eqref{flowplate} (or \eqref{withbeta} if $\beta>0$) with finite energy (flow-localized) initial data in $y_0 \in Y_{\rho_0}$. If there is a time $T^*$ so that
\begin{equation}\label{thisone}\sup_{t\in [T^*,\infty)}\left\{||u(t)||_4^2+||\Delta u_t(t)||_0^2+||u_{tt}(t)||^2_0\right\}\le C_1,\end{equation} then for any sequence of times $t_n \to +\infty$, there is a subsequence of times $t_{n_k}$ and a point $\widehat{y}=(\widehat u,0;\widehat \phi,0)$ with $(\widehat u,\widehat \phi) \in \mathcal N$ so that
$$\lim_{k \to \infty} d_{Y_{\rho}}(S_{t_{n_k}}(y_0),\widehat y) = 0$$ for any $\rho>0$. This implies that the result of Theorem \ref{conequil} then holds.
\end{theorem}

We present this theorem as an independent result, in the spirit of what is used in \cite{ryz,ryz2}.   However, a word of caution: the assumed bound (\ref{thisone}) will be valid only for ``certain" solutions to the plate problem, and certainly {\it not}  for every weak solution.  This is in striking contrast with \cite{ryz,ryz2} where  the smoothing properties of thermoelasticity provide the additional boundedness of  {\em all} plate solutions in higher topologies.  Here, we can obtain the global-in-time bounds in higher topologies (as in \eqref{thisone} below) by {\em considering smooth initial data} for the flow-plate system and {\em propagating} this regularity---see Section \ref{propagate} and \ref{highernorms000} below. Alternatively,  a key point of our proof below relies on a decomposition of the plate dynamics (see Section \ref{platedecomp}) in which one of the decomposed portions of the plate dynamics will satisfy a bound like that above in \eqref{thisone} (and the other will be exponentially stable). The non-decomposed plate dynamics will not necessarily have such a global-in-time bound. See Remark \ref{dada} below for further discussion of how Theorem \ref{highernorms0} and its proof are utilized later.

\begin{proof}[Proof of Theorem \ref{highernorms0}]
We first note that the solution to \eqref{flow1} (via classical scattering theory---see \cite{ryz,ryz2} and references therein) is stable, so $$(\phi^*(t_n),\phi^*_{t}(t_n)) \to (0,0)$$ in the local flow energy topology $\widetilde Y_{fl}$.

 We now utilize the assumption \eqref{thisone}, along with the bound in Lemma \ref{compact2}
 to bound, uniformly, the flow in higher norms. Specifically, we note that by assumption
 \begin{align}
 ~~u \in &~C^2([T^*,\infty);L_2(\Omega))\cap C^1([T^*,\infty);H_0^2(\Omega)) \cap C([T^*,\infty);H^4(\Omega)\cap H_0^2(\Omega))
 \end{align}
 (Below we utilize the notation $H^s(H^r(\Omega))$, meaning that $H^s$ regularity refers to the time variable. )

  This implies, in particular, that ~~$u\in H^2(L_2(\Omega)) \cap H^1(H^2(\Omega)) \cap L_2(H^4(\Omega)). $
 With~~  $u\in L_2(H^4(\Omega)) \cap H^1(H^2(\Omega)) $, via interpolation \cite{interp,LM}, we have
 \begin{equation}\label{1}
 u \in H^{\alpha_1}(H^{2+ 2 (1-\alpha_1) }(\Omega) ), ~~\alpha_1\in [0,1].
 \end{equation}
 
 Similarly, from $\displaystyle u_{t}\in L_2(H^2(\Omega)) \cap H^1(L_2(\Omega)) $
 we obtain
 \begin{equation}\label{2}
 u_{t}\in H^{\alpha_2}( H^{2(1-\alpha_2)}(\Omega)),~~ \alpha_2 \in [0,1]
 \end{equation}
 
 To apply inequality in Lemma \ref{compact2} with $s=0,~\beta > 0$ we will take $\alpha_1 = \beta  =\alpha_2 $
 \begin{align*}
4 -2 \alpha_1 \geq &~2+ \alpha_1;~~~ \text{so},~~~  2 \geq 3 \alpha_1~
\text{ and } ~
  2 -2 \alpha_2 \geq ~~~1 + \alpha_2 ;~~~ \text{so},~~~ 1 \geq 3 \alpha_2. \end{align*}
Hence, for
 $\eta \leq 1/3 $  we have:
 $$\|\nabla \phi^{**}(t)\|^2_{\eta, K_{\rho} }+\|\phi^{**}(t)\|^2_{\eta, K_{\rho} } \le \left\{\|u(t)\|^2_{H^{s+\eta}(t-t^*,t;H_0^{2+\eta}(\Omega))}+\|u_{t}(t)\|^2_{H^{s+\eta}(t-t^*,t;H_0^{1+\eta}(\Omega))} \right\} \le C_1,$$ where $C_1$ corresponds to the global-in-time bound on plate solutions in the assumption of the theorem. Then, applying the compactness criterion in Lemma \ref{compactnesscriterion} we have shown:
\begin{lemma}\label{limitpoint} For any sequence of times $t_n \to +\infty$, there is a subsequence (also denoted $t_n$) and a limit point $(\hat u, 0;\hat \phi, \hat \psi) \in Y$ such that:
$$(u(t_n),u_{t}(t_n);\phi(t_n)^{**},\phi_{t}^{**}(t_n)) \to (\hat u, 0; \hat \phi, \hat \psi)$$ in $\widetilde Y$ as $n \to \infty$. (Note:  we know, a priori, that $u_t \to 0$ in $L_2(\Omega)$---see Theorem \ref{convergenceprops}.)
\end{lemma}

\noindent Now, to further characterize the flow limit point, we return to the formula in Theorem \ref{flowformula}.
\begin{lemma}\label{characterize}
The limit point $\widehat \psi \in  Y$ in Lemma \ref{limitpoint} above is identified as $0$ in $L_2(K_{\rho})$ for any $\rho>0$.
\end{lemma}
\begin{proof}[Proof of Lemma \ref{characterize}] The flow solution $(\phi^*(t),\phi_{t}^*(t))$ tends to zero (again, we emphasize that this convergence is in the local flow energy sense); hence, we have that $\phi(t_n) \to \hat \phi$. To identify $\hat \psi$ with zero, we note that (by assumption) $u_{t}(t) \in H^1(\Omega)$. This allows to differentiate the flow formula in Theorem \ref{flowformula} in time (as was done in \cite{springer} and \cite{ryz2}) to arrive at that in \eqref{formderiv1}.

\noindent For a fixed $\rho >0$, we multiply  $\phi_{t}^{**}(\xb,t)$ in \eqref{formderiv1} by by a smooth function $\zeta \in C_0^{\infty}(K_{\rho})$ and integrate by parts in space---in \eqref{formderiv1} we move $\partial_x$ onto $\zeta$ in the third term and the partials $\partial_x,~\partial_y$ from $M_{\theta}$ onto $\zeta$ in term four.
This results in the bound
\begin{equation}\label{whatwe}
|(\phi^{**}_{t},\zeta)_{L_2(K_{\rho})}| \le C(\rho)\sup_{\tau \in [0,t^*]}\|u_{t}(t-\tau)\|_{0,\Omega}\|\zeta\|_{1,K_{\rho}}
\end{equation}

From this point, we utilize the fact that $u_{t}(t) \to 0$ in $L_2(\Omega)$ when $t \rightarrow \infty $  (Theorem \ref{convergenceprops}), and hence $(\phi_{t}^{**}(t),\zeta)_{K_{\rho}} \to 0$ and $\phi^*_{t} \to 0$ in $L_2(K_{\rho})$ for any fixed $\rho>0$. This gives that $\phi_{t}(t) \rightharpoonup 0$ in $L_2(\Omega)$, and we identify the strong limit point $\widehat \psi$ with $0$ in $L_2(\Omega)$ so that  $(u(t_n),u_{t}(t_n);\phi(t_n)^{**},\phi_{t}^{**}(t_n)) \to (\hat u, 0; \hat \phi, 0)$ in $\widetilde Y$ as $n \to \infty$. . \end{proof}

Now we show that the limit points obtained above satisfy the static problem in a weak sense.
\begin{lemma}\label{staticsols}
The pair $(\hat u, \hat \phi)$ as in Lemma \ref{limitpoint} and Lemma \ref{characterize} satisfies the stationary problem \eqref{static} in the variational sense.
\end{lemma}
\begin{proof}[Proof of Lemma \ref{staticsols}] We begin by multiplying the system \eqref{flowplate} by $\eta \in C_0^{\infty}(\Omega)$ (plate equation) and $\psi \in C_0^{\infty}(\realsthree_+)$ (flow equation) and integrate over the respective domains. This yields
\begin{equation}\begin{cases}
\lb u_{tt},\eta\rb +\lb \Delta u, \Delta \eta\rb +k\lb u_t,\eta\rb +\lb f(u),\eta\rb  = \lb p_0,\eta\rb -\lb tr[\phi_t+U\phi_x],\eta\rb  \\
\left(\phi_{tt},\psi\right)+U(\phi_{tx},\psi)+U(\phi_{xt},\psi)-U^2(\phi_{x},\psi_x)=-(\nabla \phi, \nabla \psi)+\lb u_t+Uu_x, tr[\psi]\rb .
\end{cases}
\end{equation}

Now, we consider the above relations evaluated at the points $t_n$ (identified as a subsequence for which the various convergences above hold), and integrating in the time variable from $t_n$ to $t_n+c$. Limit passage on the linear terms is clear, owing to the main convergence properties for the plate component in the Theorem \ref{convergenceprops}. The locally Lipschitz nature the nonlinearities of physical type allows us to pass with the limit on the nonlinear term (this is by now standard, \cite{springer}). We then arrive at the following static relations:
\begin{equation}\begin{cases}
\lb \Delta \hat u, \Delta w\rb+\lb f(u),w\rb = \lb p_0,w\rb+U\lb tr[\hat \phi],w_x\rb \\
(\nabla \hat \phi, \nabla \psi)-U^2(\hat \phi_{x},\psi_x)=U\lb \hat u_x, tr[\psi]\rb.
\end{cases}
\end{equation}
This implies that our limiting point $(\hat u,0; \hat \phi,0)$ of the sequence \newline $\left(u(t_n),u_t(t_n);\phi(t_n),\phi_t(t_n)\right)$ is in fact a solution (in the weak/variational sense) to the static equations, i.e., it is a {\em stationary solution} to the flow-plate system \eqref{flowplate}. We have thus shown that any trajectory contains a sequence of times $\{ t_n\}$ such that, along these times, we observe convergence to a solution of the stationary problem. \end{proof}
The above lemmata complete the proof of Theorem \ref{highernorms0}.

\end{proof}
\begin{remark}\label{dada}
In the proof above we have shown how a bound on the plate dynamics in higher topologies results in the desired convergence to equilibrium result for the full dynamics. We pause to point out that this will be used in two different ways: below we will show that with {\em smooth initial data} for \eqref{flowplate}, via a propagation of regularity result, such a bound can be obtained. Alternatively, in the proof of our main result (in Section \ref{flowdecomp}), we will use the proof of Theorem \ref{highernorms0} for decomposed systems \eqref{wfull} and \eqref{zfull}. We note that the result of Theorem \ref{highernorms0} is not strictly dependent on the structure of the plate equation, but rather the linear nature of the coupling via the Neumann condition for the flow \eqref{flowformula}. Hence, if we add change slightly the structure of the plate equation,  the result of Theorem \ref{highernorms0} will be unchanged, so long as (i) the modified plate dynamics still converge in the sense of Theorem \ref{convergenceprops}, and (ii) the requisite bound in Theorem \ref{highernorms0} is obtained. Specifically, in what follows we will consider a general system: 
\begin{equation}\label{wfull*}\begin{cases}
u_{tt} + (k+1) u_t + \Delta^2 u  = p_0-U\partial_x u-q^u+F(t,u)~~ &\text{ in } \Omega\\
u=\Dn u=0~~\hskip1cm\text{ on } \Gamma\\
u(t_0) = u_0, ~~u_t(t_0) = u_1,~~u^{t_0}=\eta,
\end{cases}
\end{equation}
where the data will be ``smooth" (as in the proof of Theorem \ref{dichotoo1}), or possibly zero (as in the proof of Theorem \ref{dichotoo2}). In both of these cases the nonlinearity $F(t,u)$ will be broader than $f=f(u)$, depending on a {\em given function}, owing to the decomposition of the dynamics considered below in \eqref{exp} and \eqref{smooth}. 
\end{remark}

\subsection{Propagation of Regularity}\label{propagate}
In order to utilize the key technical supporting result in Theorem \ref{highernorms0}, we must show a propagation of regularity result for the dynamics. Specifically, we must show that for ``smooth" initial data, we have ``smooth" time dynamics; this is Theorem \ref{highernorms}. Such a result was correctly stated in the \cite[(2014)]{conequil1}, however the proof was unclear. This has been updated in a subsequent version \cite[(2015)]{conequil1} of the same reference. We include a proof here in the case of Berger dynamics (with the clear proof for von Karman dynamics in \cite[(2015)]{conequil1}). 

\begin{remark}\label{propagate*} To show infinite-time propagation of the plate dynamics, we rely on the full flow-plate dynamics to achieve {\em propagation of the initial regularity} on any interval $[0,T^*]$. Once this is achieved, we may work on the reduced delay plate (after sufficient time has passed) and utilize sharp bounds to obtain regularity of plate trajectories on the infinite horizon (propagation on $[T^*,\infty)$).
This part of the argument depends on  (i) uniform exponential decay for the nonlinear plate equation with large static and viscous damping \cite{lagnese}, and (ii) specific properties of nonlinearity; (in the von Karman case one applies  the {\it sharp} regularity of Airy stress function \cite[p. 44]{springer}. \end{remark}

\begin{theorem}\label{highernorms}  Consider the dynamics $(S_t,Y)$ corresponding to \eqref{flowplate}. Consider initial data $y_m \in \cD(\bT)$ such that $y^m \in B_R(Y)$ and take $k>0$. Then we have that for the trajectory $S_t(y^m) = (u^m(t),u^m_{t}(t);\phi^m(t),\phi^m_{t}(t))$ $$(u^m(\cdot),\phi^m(\cdot)) \in  C^1\left(0,T;H_0^2(\Omega)\times H^1(\realsthree_+)\right),$$ for any $T$,
along with the bound
\begin{equation}\label{keybound0}
\sup_{t\in[0,T)}\left\{\|\Delta u^m_{t}\|^2_{\Omega}+\|u^m_{tt}\|^2_{\Omega} \right\}  \leq C_{m,T} < \infty.
\end{equation}
Additionally, if we assume the flow initial data are localized and consider the delayed plate trajectory (via the reduction result Theorem \ref{rewrite}), we will have:
\begin{equation}\label{keybound1}
\sup_{t\in[0,\infty)}\left\{\|\Delta u^m_{t}\|^2_{\Omega}+\|u^m_{tt}\|^2_{\Omega}  \right\} \le C_1< \infty.
\end{equation}
Then, by the boundedness in time of each of the terms in the \eqref{reducedplate}, we have
$$\sup_{t \in [0,\infty)} \|\Delta^2u^m(t)\|_0 \le C_2,$$ where we critically used the previous bound in \eqref{keybound1}. In particular, this implies that, taking into account the clamped boundary conditions: \begin{equation}\label{keybound2}
\sup_{t \in [0,\infty)}\|u^m(t)\|_{4,\Omega} \le C_3.
\end{equation}
Each of the terms $C_i$ above depends on the intrinsic parameters in the problem (including the respective loading, which depends on the nonlinearity being considered).
\end{theorem}
\begin{remark}
The superscript $m$ in the above theorem serves to indicate that ``smooth" data (in the sense above) will be used to approximate given finite energy data. 
\end{remark}
\begin{remark}
The parameter $\beta \ge 0$ is inert with respect to this result. That is, Theorem \ref{highernorms} holds for \eqref{flowplate} ($\beta =0$) and \eqref{withbeta} ($\beta>0$).
\end{remark}
For this proof we consider no static damping ($\beta=0$ in \eqref{withbeta}) and any $k>0$ fixed. Additionally, this propagation result holds for both the Berger and von Karman nonlinearities \cite{conequil1}.

The proof of this Theorem \ref{highernorms} is based on two steps labeled STEP 1 and STEP 2. We first prove finite time propagation of regularity (STEP 1). Following this, we shall propagate regularity
for all times uniformly (STEP 2) {\em without making assumptions on the damping coefficients}. This is motivated by the fact that for the infinite time propagation we must use the delay representation for plate solutions, which is valid only for sufficiently large times. Thus, the regular initial condition required to do so is obtained via the finite time propagation of regularity for the full flow-plate model.

\begin{proof}[Proof of Theorem \ref{highernorms}] ~~{\phantom{aaaa}}
\subsubsection{STEP 1}
To prove \eqref{keybound1} above (which then implies \eqref{keybound2}) we will consider the time differentiated version of the entire flow-plate dynamics in \eqref{flowplate}, which is permissible in $\cD(\bT)$. We label $\overline{u}=u^m_{t}$ and $\overline{\phi}=\phi^m_{t}$, multiply the plate equation by $\overline u_t$ and the flow equation by $\overline \phi_t$, and integrate in time. We consider the `strong' energies
\begin{align}
E^s_{pl}(t)=& ~\dfrac{1}{2} [\|\Delta \ou(t)\|^2+\|\ou_t(t)\|^2],\notag\\
E^s_{fl}(t)=& \dfrac{1}{2}[\|\nabla \of(t)\|^2_{\realsthree_+}-U^2\|\partial_x \of(t)\|_{\realsthree_+}^2+\|\of_t(t)\|^2_{\realsthree_+}],\notag\\
E^s_{int}(t)=&~2U\lb \ou_x(t),\of(t)\rb_{\Omega}\notag\\
\cE^s(t) = &~E^s_{pl}(t)+E^s_{fl}(t)+E^s_{int}(t),~~
\bE^s(t) =  E^s_{pl}(t)+E^s_{fl}(t)
\end{align}
From the standard analysis (done at the energy level) we have that \cite{springer,webster}
\begin{align}
|E^s_{int}(t)| \le \delta \|\nabla \phi^m_{t}(t)\|^2&+\frac{C}{\delta}\|u^m_{xt}(t)\|_0^2 \le \delta \|\nabla \phi^m_{t}(t)\|^2+\epsilon \|\Delta u^m_{t}(t)\|^2+C_{\epsilon,\delta}\|u^m_{t}(t)\|^2
\end{align}
Using the boundedness of trajectories at the energy level $\sup_{t>0}\|u^m_{t}(t)\|_0 \le C_R$, we then have that for all $\epsilon>0$
\begin{equation}
|E^s_{int}(t)| \le \epsilon\bE^s(t)+K_{\epsilon},
\end{equation}
and this results in the existence of constants $c,C,K>0$ such that \begin{equation}\label{equiv}
c\bE^s(t)-K \le \cE^s(t) \le C\bE^s(t)+K, ~~\forall t>0.
\end{equation}
Then the energy balance for the time-differentiated equations is:
\begin{equation}\label{strongbalance}
\cE^s(t) +k\int_0^t \|\ou_t(\tau)\|^2 d\tau = \cE^s(0)+B(t),
\end{equation}
where $B(t)$ represents the contribution of the nonlinearity in the time-differentiated dynamics:
\begin{equation}\label{bee}
B(t) \equiv \int_0^t \left\lb \dfrac{d}{dt}(f(u^m)),\ou_t\right\rb_\Omega d\tau
\end{equation}
Using \eqref{equiv}, we have
\begin{equation}
\bE^s(t) \le C\bE^s(0)+K+|B(t)|
\end{equation}
We note immediately that if we can bound $B(t)$, we will obtain the estimate in \eqref{keybound1}. Continuing, we use the key decomposition 
 for $f(u)=[b - ||\nabla u||^2]\Delta u$:
\begin{align}
\left(\dfrac{d}{dt} \left([||\nabla u^m||^2-b]\Delta u^m\right),u^m_{tt}\right)_{\Omega}=&~2(\nabla u^m,u^m_{t})(\Delta u^m,u^m_{tt})+[||\nabla u^m||^2-b](\Delta u^m_{t}, u^m_{tt})\\
=&~ \dfrac{d}{dt}Q_1(t)+P_1(t)
\end{align}
where 
\begin{align}
Q_1(t) = & ~-\dfrac{1}{2}[||\nabla u^m||^2-b]||\nabla u^m_{t}||^2-(\Delta u^m,u^m_{t})^2, \\
P_1(t) = & ~-3(\Delta u^m,u^m_{t})||\nabla u^m_{t}||^2.
\end{align}

Thus we have
\begin{align}
|B(t)| \le &~ C\Big\{\|~\big[||\nabla u^m||^2-b\big]||\nabla u^m_{t}||^2-(\Delta u^m,u^m_{t})^2\|_0\|u^m_{t}\|_0\Big|_0^t+\int_0^t|(\Delta u^m,u^m_{t})|~||\nabla u^m_{t}||^2 d\tau \Big\}\notag\\
\le &~\epsilon ||u^m_{t}(t)||_2^2 +C_{\epsilon}(R)||u^m_{t}||^2+C(\bE^s(0)) + C_R\int_0^t\|u^m_{t}\|_2^2 d\tau
\end{align}
We again appeal to the hyperbolic estimate:
$$||u^m_{t}(t)||^2_0 \le ||u^m_{t}(0)||^2_0 +\int_0^t||u^m_{tt}(\tau)||_0^2 d\tau,$$ which yields:
\begin{align}
|B(t)| \le &~\epsilon ||u^m_{t}(t)||_2^2 +C_{\epsilon,R}(\bE^s(0)) + C_{\epsilon,R}\int_0^t[\|u^m_{t}\|_2^2+\|u^m_{tt}\|_0^2] d\tau \\
\le &~ ||u^m_{t}(t)||_2^2 +C_{\epsilon,R}(\bE^s(0)) + C_{\epsilon,R}\int_0^t\mathbf E^s(\tau) d\tau
\end{align}

 Hence the above inequality, when used in conjunction with \eqref{equiv}, \eqref{strongbalance}, and Gronwall's inequality, provide uniform boundedness of the higher norm energy $\bE^s(t)$ on any finite time interval $[0,T]$, for  any {\it finite }  $T$.

\subsubsection{STEP 2}
 \begin{remark}
Using the variation of parameters formula, and global-in-time bounds of delayed plate trajectories in the energy topology, one can obtain the bounds in \eqref{keybound1} and \eqref{keybound2} straightforwardly, however this requires {\em large viscous damping} (see \eqref{varpar0}). Here we show below that these bounds (and the final result for strong stability for smooth initial data) can be obtained with minimal viscous damping and NO static damping. 
\end{remark}
 In order to obtain uniform-in-time boundedness for {\em plate solutions} (as in \ref{keybound1} and \ref{keybound2}), more work is needed. The needed argument depends on the reduction in Theorem \ref{rewrite}, where the flow dynamics are reduced
to a non-dissipative and a delayed term in the plate dynamics. This reduction is valid for sufficiently large times---thus the result of STEP 1 applies to all the times before the delayed model is valid.

We return to the delay representation of the  plate solutions (Theorem \ref{rewrite}) and utilize our finite time horizon regularity from STEP 1 to obtain regular {\em plate data} at time $T$. Considering $(\ou,\ou_t)=(u^m_{t},u^m_{tt})$, we write
\begin{equation}\label{w}
\ou_{tt} + \Delta^2 \ou + (k + 1) \ou_t  =\frac{d}{dt} \left\{[||\nabla u^m||^2-b]\Delta u^m \right\}- U \partial_x \ou - \frac{d}{dt} \left\{q(u^m)\right\}
\end{equation} which is valid for any $t>\max \{T,t^{\#}\}$.

We utilize an approach taken in \cite{springer} for the non-delayed plate. We will define a Lyapunov-type function for the time-differentiated, smooth dynamics so that we may use Gronwall, along with the finiteness of the dissipation integral (Corollary \ref{dissint}):
$$\int_0^{\infty}||u^m_{t}||^2 d\tau = \int_0^{\infty} ||\ou||^2d\tau<+\infty.$$ 

\noindent Let us use some more convenient notation here: $E(t)=\dfrac{1}{2}[||\Delta \ou||^2+||\ou_t||^2]$. Then define $$Q(t) = E(t)-Q_1(t)+\nu||\ou||^2,$$
where $Q_1$ and $P_1$ are as above.

Let $\nu,\mu >0$. We can choose $\nu(R)$ large enough so that $Q(t)$ is positive \cite{springer}; we make this choice and then keep $\nu$ fixed. 
We will now construct a Lyapunov-like function for the dynamics $T^{\ou}_t(x_0) = (\ou(t),\ou_t(t);\ou^t) \in \mathbf H$ with $x_0 \in \mathbf H$. There are constants $a_0(\nu)$ and $a_1(\nu)$ so that:
$${a_0}\Big[||\Delta \ou||^2+||\ou_t||^2\Big]\le  Q(t) \le {a_1}\Big[||\Delta \ou||^2+||\ou_t||^2\Big].$$
The parametric Lyapunov function is then:
$$W(t) = Q(t)+\varepsilon (\ou_t,\ou)+\mu\int_0^{t^*}\int_{t-s}^tE(\tau)d\tau ds,~~\varepsilon \in (0,\varepsilon_0).$$
The parameter $\mu$ will be addressed below.
The value of $\varepsilon_0$ is chosen so that 
\begin{equation}\label{abovebelow} \dfrac{1}{2}a_0 E(t) \le W(t) \le 2a_1E(t)+\mu t^*\int_{t-t^*}^tE(\tau),~~\forall~~\varepsilon \in (0,\varepsilon_0),\end{equation}
 and we note that 
$$\mu \int_0^{t^*}\int_{t-s}^t E(\tau) d\tau ds \le \mu t^* \int_{t-t^*}^t E(\tau).$$
We compute $W'(t)$:
\begin{align}\label{w1}
\dfrac{d}{dt} W(t) =&~ \langle \Delta \ou,\Delta \ou_t\rangle +\langle \ou_t,\ou_{tt}\rangle -\dfrac{d}{dt}Q_1(t)+\nu\langle \ou,\ou_t\rangle  \\
&+\varepsilon \langle \ou_{tt},\ou\rangle +\varepsilon ||\ou_t||^2 +\mu t^*E(t)-\mu \int_{t-t^*}^t E(\tau)
\end{align}
Make the replacement:
$$\ou_{tt} = - \Delta^2 \ou-(k+1)\ou_t -\dfrac{d}{dt}q^{u^m}-\dfrac{d}{dt}[f(u^m)]-U\ou_x.$$
Then:
\begin{align*}
\dfrac{d}{dt} W(t) =&~(\mu t^*+\varepsilon -k-1)||\ou_t||^2+(\nu-\varepsilon k-1)\langle \ou,\ou_t\rangle -\langle [q^{u^m}]',\ou_t\rangle\\\nonumber& -\langle [f(u^m)]',\ou_t\rangle -U\langle \ou_x,\ou_t\rangle -\dfrac{d}{dt}Q_1 \\
&(\mu t^*-\varepsilon)||\Delta \ou||^2+\varepsilon\langle [q^{u^m}]',\ou\rangle -\varepsilon\langle [f(u)]',\ou\rangle -\varepsilon U\langle \ou_x,\ou\rangle -\mu \int_{t-t^*}^t E(\tau) \end{align*}

Thus:
\begin{align}
\dfrac{d}{dt}W(t)&+(k+1)||\ou_t||^2+\varepsilon||\Delta \ou||^2+\mu \int_{t-t^*}^t E(\tau)d\tau  \\\nonumber
=&~(\mu t^*+\varepsilon)||\ou_t||^2+\mu t^* ||\Delta \ou||^2+P_1(t)-\varepsilon\langle [f(u^m)]',\ou\rangle \\\nonumber
&-{\varepsilon\langle [q^{u^m}]',\ou\rangle-\langle [q^{u^m}]',\ou_t\rangle }+(\nu-\varepsilon(k+1))\langle \ou,\ou_t\rangle -\varepsilon U\langle \ou_x,\ou\rangle-U\langle \ou_x,\ou_t\rangle ,
\end{align}
Via \eqref{abovebelow}, choosing $\mu$ and $\varepsilon$ sufficiently small, there is a $\gamma(a_1,\nu,k,\mu,t^*)>0$ so that for any $\varepsilon \in (0,\varepsilon_0)$:
\begin{align}\label{zzz}
\dfrac{d}{dt}W(t)&+\gamma [E(t)+\int_{t-t^*}^t E(\tau) d\tau ] \le  C\Big\{ P_1(t)-\varepsilon\langle [f(u)]',\ou\rangle -\varepsilon\langle[q^{u^m}]',\ou\rangle
\\\nonumber&-\langle [q^{u^m}]',\ou_t\rangle+(\nu-\varepsilon k-1)\langle \ou,\ou_t\rangle-\varepsilon U\langle \ou_x,\ou\rangle -U\langle \ou_x,\ou_t\rangle\Big\}
\end{align}
\begin{lemma}\label{lyapo}
There exists a $\widetilde \gamma>0$ (with the same dependencies $\gamma$  above), and constant $M$ (depending on $R$ and the same quantities as $\widetilde \gamma$) so that:
\begin{equation}
W'(t)+\widetilde \gamma W(t)  \le M[1+||\ou||^2E(t)]
\end{equation}
\end{lemma}
\begin{proof}[Proof of Lemma \ref{lyapo}]

We recall \eqref{qderiv}. So, along with the global-in-time priori bounds for the lower energy energy level $(u^m,u^m_t)$,
$$\sup_{t \in [t_0,\infty)} \Big[||u^m||^2_2+||\ou||^2 \Big]<\infty,$$
and 
\begin{align*}\dfrac{d}{dt}f(u^m)=&~-2(\nabla u^m,\nabla u^m_t)\Delta u^m+[b-||\nabla u^m||^2]\Delta u^m_t\\
=&~2(\Delta u^m, \overline u)\Delta u^m+[b-||\nabla u^m||^2]\Delta \overline u,\end{align*}
we obtain:
\begin{align}
|P_1| \le &~C||\Delta u^m||~||\ou||~||\nabla \ou||^2 \\[.2cm]
 |\langle [f(u^m)]',\ou\rangle | \le &~ C(R)\big[||\ou||^2+||\ou||~||\Delta \ou||\big]\\[.2cm]
 (\nu-\varepsilon k-1)|\langle \ou, \ou_t\rangle | \le &~\delta ||\ou_t||^2+C_{\delta}(\nu,\varepsilon,k,R)\\[.2cm]
 \varepsilon U|\langle \ou_x,\ou\rangle+U|\langle \ou_x,\ou_t\rangle | \le &~\delta[ ||\ou_t||^2+||\Delta \ou||^2] +C_{\delta}(\varepsilon,U,R)\\[.2cm]
\varepsilon |\langle [q^{u^m}]',\ou \rangle| \le&~ C(\varepsilon)\left(||u^m(t)||^2_{2}+||u^m(t-t^*)||^2_{2}+\int_{-t^*}^0||u^m(t+\tau)||^2_2 d\tau\right) \\\nonumber
&+\delta ||\ou||_2^2+C_{\delta}(\varepsilon) ||\ou||^2 \\[.2cm]
\le &~\delta E(t) +C_{\delta}(\varepsilon,t^*,R)
\end{align}
And, noting the explicit estimates for the term $q_t^u$ performed in \cite[(6.1)]{delay}, we also have
\begin{align}
\langle \partial_t [q^{u^m}], \ou_t \rangle \leq \delta ||\ou_t||^2 +C_{\delta} \Big(\|u^m(t)\|^2_{2,\Omega}
 + \| u^m(t-t^*)\|^2_{2,\Omega} \\\nonumber+
  \int_{-t^*}^0 \|u^m(t+\tau)\|^2_{3,\Omega} ) d \tau\Big).
\end{align}

\noindent Then we note: $$\|u\|^2_{3} \leq \delta \|u\|^2_4 + C_{\epsilon}(R)\|u\|^2_2 \leq C_{\epsilon}  + \epsilon \|u\|^2_4, $$
where $\|u\|_4 $ is to be calculated from an elliptic equation in terms of $\|(\ou,\ou_t) \|_{{Y}_{pl}}$.
Indeed, we consider biharmonic problem with  the clamped boundary conditions.
$$ \Delta^2 u^m = - \ou_t - (k+1) \ou  + [||\nabla u^m||^2-b]\Delta u^m -Uu^m_{x} - q^{u^m} $$
with
$$u^m = \Dn u^m =0 ~\text{on}~ \partial \Omega.$$
This gives
\begin{align*}\|u^m\|_4 \leq& ~C  \|\ou_t - (k+1) \ou  + [||\nabla u^m||^2-b]\Delta u^m -u^m_{x}  - q^{u^m}\|
\\
\leq&~ C  \left[\|(\ou,\ou_t)\|_{Y_{pl}} +{ \|u^m\|_1 + \|u^m\|^3_2 + \|u^m\|^2_2 }\right].
\end{align*}
 We have
$$\|u^m\|^2_4 \leq  C(R) [E(t) + {1}].$$

\begin{equation}
\langle [q^{u^m}]',\ou_t \rangle \le ||\ou_t||_0 \left|\left|\dfrac{d}{dt} q^{u^m} \right|\right|_0 \le \delta \left[E(t)+\int_{t-t^*}^t E(\tau) d\tau\right] + C_{\delta}(t^*,R)
\end{equation}

Combining these estimates, we return to  \eqref{zzz}:
\begin{align}
\dfrac{d}{dt}W(t)+\gamma [E(t)+\int_{t-t^*}^t E(\tau) d\tau]=&~  C\Big\{ P_1(t)-\varepsilon([f(u^m)]',\ou)-{\varepsilon([q^{u^m}]',\ou)-([q^{u^m}]',\ou_t)}
\\\nonumber&+(\nu-\varepsilon (k+1))(\ou,\ou_t)-\varepsilon U(\ou_x,\ou)-U(\ou_x,\ou_t)\Big\} \\
\le &~ \delta [E(t)+\int_{t-t^*}^t E(\tau)d\tau] +C_{\delta}(\varepsilon,\nu,R) ||\ou||^2E(t) \\\nonumber
&+ C_{\delta}(R,\nu,\varepsilon,t^*,U,k)
\end{align}
Scaling $\delta$ appropriately and absorbing terms, we have:
$$\dfrac{d}{dt}W(t)+\gamma [E(t)+\int_{t-t^*}^t E(\tau) d\tau] \le M+C||\ou||^2E(t).$$
Recalling \eqref{abovebelow}, we have some $\widetilde \gamma>0 $ such that:
$$\dfrac{d}{dt}W(t)+\widetilde \gamma W(t) \le M+C||\ou||^2E(t).$$
\end{proof}
Noting \eqref{abovebelow}, we have:
$$\dfrac{d}{dt}W(t)+\widetilde \gamma W(t) \le M+C_1||\ou||^2 W(t),$$
and utilizing the integrating factor we see that
$$\dfrac{d}{dt}\left(e^{\widetilde \gamma t}W(t) \right) \le Me^{\widetilde \gamma t}+C_1e^{\widetilde \gamma t}||\ou||^2W(t).$$
Integrating, we have:
$$e^{\widetilde \gamma T} W(T) \le e^{\widetilde \gamma t_0}W(t_0)+\dfrac{M}{\widetilde \gamma}\left[e^{\widetilde \gamma T}-e^{\widetilde \gamma t_0}\right]+C_1\int_{t_0}^Te^{\widetilde \gamma t}||\ou(t)||^2W(t)dt,$$
and thus
\begin{align*} W(T) \le&~ e^{-\widetilde \gamma(T-t_0)} W(t_0)+\dfrac{M}{\widetilde \gamma}\left[1-e^{-\widetilde \gamma (T-t_0)}\right]+C_1\int_{t_0}^T||\ou(t)||^2W(t)dt \\
\le &~\dfrac{M}{\widetilde \gamma}+e^{-\widetilde \gamma(T-t_0)} \left[W(t_0)-\dfrac{M}{\widetilde \gamma}\right]+C_1\int_{t_0}^T||\ou(t)||^2W(t)dt.
\end{align*}
 Let $$\alpha(T) = \dfrac{M}{\widetilde \gamma}+e^{-\widetilde \gamma(T-t_0)} \left[W(t_0)-\dfrac{M}{\widetilde \gamma}\right],$$ and note that $\alpha(T) \ge 0$ for $T\ge t_0$. Then, using Grownall, we have:
$$W(T) \le \alpha(t)+\int_{t_0}^T \alpha(t)||\ou(t)||^2\exp\left\{\int_t^T||\ou(s)||^2ds\right\} dt.$$ Note that $\alpha(T) \ge 0$ and $\alpha$ is bounded, and thus, owing to the finiteness of the dissipation integral $$\int_{t_0}^{\infty}||\ou||^2 dt=\int_{t_0}^{\infty} ||u^m_{t}||_0^2dt<+\infty,$$ we obtain global-in-time boundedness of $W(t)$. Then, using \eqref{abovebelow}, we have global-in-time boundedness of $$E(t)=\dfrac{1}{2}\big[||\Delta u^m_{t}||^2+||u^m_{tt}||^2\big].$$Via elliptic theory (as utilized above), we also have 
$||u^m(t)||_4^2 $ globally bounded in time. We note that $W(t_0)$ has the requisite regularity by STEP 1 above (where we have shown propagation of regularity on the finite time horizon for the full flow-plate dynamics).

This concludes the proof of Theorem \ref{highernorms}.
\end{proof}

\begin{remark}
As mentioned above, the proof steps can be easily adapted to the case of the von Karman nonlinearity using the {\em sharp regularity of the Airy Stress function} \cite[p. 44]{springer}.
\end{remark}

\subsection{Strong Convergence for Smooth Data}\label{highernorms000}
From Theorem \ref{highernorms} we see that for the Berger dynamics (and, in fact for von Karman), in the presence of minimal viscous damping (only requiring $k>0$), we have a propagation of regularity result on the infinite-time horizon. From this bound, for initial data $y^m \in \mathscr D(\mathbb T)$, we can immediately apply Theorem \ref{highernorms0}. This yields one of the primary results of \cite[Theorem 2.1]{conequil1}:

\begin{theorem}\label{regresult} 
Consider $y_m \in \cD(\bT)\cap B_R(Y) \subset (H^4\cap H_0^2)(\Omega)\times H_0^2(\Omega) \times H^2(\realsthree_+) \times H^1(\realsthree_+) \subset Y$. Suppose that the initial flow data are supported on a ball of radius $\rho_0$ (as in Theorem \ref{rewrite}). Suppose $k>0$. Then for any sequence of times $t_n \to \infty$ there is a subsequence of times $t_{n_k}$ identified by $t_k$ and a point $\hat y =(\hat u, 0; \hat \phi, 0) \in Y$ , with $(\hat u, \hat \phi) \in \mathcal N$, so that $$\|S_{t_k}(y^m) - \hat y\|_{Y_{\rho}} \to 0,~~t_k \to \infty.$$    \end{theorem} \noindent This implies the primary result: strong convergence of $(u(t),\phi(t))$ to the equilibria set $\mathcal N$ for initial data $y_0 \in \mathscr D(\mathbb T) \cap Y_{\rho_0}$.

\begin{remark}\label{passing}
We note that the situation when stability to the set of equilibria is achieved for smooth initial data is typical in the analysis of strong stability for hyperbolic models with damping. The next step is to extend the desired property via density to all initial data residing in the finite energy space. This can be done (see, e.g., \cite{ball}) via a uniform-in-time Lipschitz (or Hadamard) property of the semigroup. As we describe in the next section, this is precisely the key issue to contend with in the analyses in \cite{conequil1} (where large static damping was used in conjunction with large static damping) and here, where a fortuitous decomposition allows us to consider only large viscous damping.
\end{remark}

		\subsection{ Strong Convergence for Finite Energy Data---Technical Difficulties} \label{techdiff}

	With the technical background established in Sections \ref{techpres}--\ref{propagate}, we now address the specific difficulties involved in showing a stabilization to equilibria result without assuming either (i) terms of the form $k_1\Delta u_{tt}-k_2\Delta u_t$ (as in \cite{chuey,springer}), or (ii) exploiting parabolic effects in a thermoelastic plate \cite{ryz,ryz2}. In both cases, as well as in \cite{delay}, the key task is {\em to first show compact end behavior for the plate dynamics}. This requires the use of a reduction of the full flow-plate dynamics to a delayed plate equation (Theorem \ref{rewrite} above), at which point one may work on this delayed plate system. In both case (i) and (ii)  above the ultimate character of the nonlinear component of the model is {\em compact}---owing to the fact that parabolic smoothing {\em and} rotational inertia both provide $\nabla u_t \in L_2(\Omega)$. This permits compactness of the boundary-to-flow (Lopatinski or hyperbolic Neumann) mapping. The results in \cite{delay} were the first to show that dissipation could be harvested from the flow in order to show ultimate compactness of the plate dynamics {\em without} utilizing the compactness of this mapping.

We note that in \cite{ryz2} the key to the stabilization result lies in a compactness criterion (given as Lemma 10 in \cite[p. 472]{ryz}) for flow trajectories (given here as \eqref{compact2}). The flow bound in higher norms is obtained via the thermoelastic character of the structural dynamics. 

In \cite{conequil1}, the major contribution was the ability to circumvent the seeming lack of natural compactness in the dynamics (for von Karman or Berger) by utilizing large static {\em and} viscous damping. The requirement for large static and viscous damping was needed to show exponential decay of plate trajectories, which resulted in uniform-in-time Hadamard continuity of the full flow-plate dynamics (see Theorem \ref{strongcont} and Remark \ref{passing} above). With this, and the desired result holding for regular initial data (Theorem \ref{regresult}), an approximation argument was used to pass convergence properties from regular data onto finite energy initial data. The principal result there (\cite[Theorem 2.3]{conequil1}) is the analog of  Theorem \ref{conequil}, however, {\em it requires  the addition of large static damping}.
\begin{theorem}\label{conequil1*}
Consider \eqref{withbeta}. Let $0\le U<1$, and assume $p_0 \in L_2(\Omega)$. Then there are minimal damping coefficients $k_{min}$ and $\beta_{min}$ so that for $k\ge k_{min}>0$ and $\beta \ge \beta_{min}>0$  any generalized solution $(u(t),\phi(t))$ to the system with localized (in space) initial flow data in $Y_{\rho}$ has the property that \begin{align*}\lim_{t \to \infty} \inf_{(\hat u,\hat \phi) \in \mathcal N}\left\{\|u(t)-\hat u\|^2_{H^2(\Omega)}+\|u_t(t)\|^2_{L_2(\Omega)}+\|\phi(t)-\hat \phi\|_{H^1( K_{\rho} )}^2+\|\phi_t(t)\|^2_{L_2( K_{\rho} )} \right\}\\=0, \text{ for any }  \rho>0.\end{align*} \end{theorem}
As in Theorem \ref{conequil} here, the minimal damping coefficients $k_{min}$ and $\beta_{min}$ depend on an invariant set for the plate dynamics, which itself depends on the intrinsic quantities in the problem. 

Though the results of \cite{conequil1} stand as the culmination of the treatments \cite{delay,webster}, and critically use the results of \cite{springer,ryz,ryz2}, static damping was necessary to contend with lower order terms (due to the non-dissipative nature of the reduced plate dynamics \eqref{gendelay}), but seemed physically anomalous. 
The key technical point in this treatment is {\em the elimination of this need for large static damping imposed on the structure}. We seek the physical result which says that for large viscous damping only (size dependent upon the intrinsic parameters of the problem) the full-flow plate dynamics stabilize to the stationary set. Our approach here is based on a decomposition of plate dynamics into an exponentially stable component and a smooth component. To show smoothness of this latter component in the decomposition, we require working in {\em higher topologies}.

\section{Proper Proof of Main Results}\label{proofmain}
In this section we prove Theorem \ref{conequil}. The proof follows through several steps and critically uses the results and arguments presented up to this point. 
\subsection{Outline of the Proof}
For the reader's convenience we now give a  summary of each of the key steps in the proof:

In Section \ref{platedecomp} we consider a decomposition of the delayed plate dynamics \eqref{reducedplate}: we consider a $z$ and $w$ component, such that $z+w=u$. We will decompose the nonlinearity in an advantageous way (for time sufficiently large), as well as add static damping $\beta z$ to the $z$ equation and subtract it from the $w$ dynamics. This allows us to use damping of the form $kz_t+\beta z$ in the $z$ plate dynamics, which will result in uniform exponentially stability for \eqref{exp}. For the $w$ component, we will consider zero initial data  (the initial data being left in the $z$ dynamics), giving \eqref{smooth}. The $w$ dynamics, with zero initial data, will be smooth, having appropriate global-in-time bounds.

Section \ref{expsec} and \ref{regsec} provide the proofs associated to the aforementioned plate decomposition and are the major technical contribution of this work. Section \ref{expsec} follows \cite{conequil1} in utilizing Lyapunov methods to show exponential stability of $z$ component of the plate decomposition; Section \eqref{regsec} demonstrates the smoothness of the $w$ component of the decomposed plate dynamics, and critically uses the null initial data in differentiating the dynamics and obtaining global-in-time bounds. 

In Section \ref{flowdecomp} we decompose the {\em flow} according the plate decompositions in the previous step. This allows us to reconstruct a full flow-plate dynamics corresponding to $z$, and a full flow-plate dynamics corresponding to $w$, coupled through a mixed, nonlinear plate term. The $(z,\phi^z)$ dynamics (corresponding to the exponentially stable $z$ plate dynamics) will exhibit global-in-time Hadamard continuity (depending critically on large static and viscous damping), and thus we can approximate any trajectory by smooth data  (this follows the general structure of the argument in \cite{conequil1}). The $(w,\phi^w)$ dynamics (corresponding to the plate dynamics with zero initial data) are smooth, and thus the results and techniques for smooth initial data can be re-applied (Theorems \ref{highernorms0} and \ref{highernorms}).

\subsection{Decomposition of Delay Plate Dynamics}\label{platedecomp}
	We consider the reduced, delay plate equation in \eqref{reducedplate} (see \cite{oldchueshov1,springer,delay}). The more general PDE system with delay \eqref{gendelay} is well-posed on $\mathbf H = H_0^2(\Omega)\times L_2(\Omega) \times L_2\left(-t^*,0;H_0^2(\Omega)\right).$  We write $T_t(x_0) = (u,u_t;u^t)$ with $x_0 =(u(t_0),u_t(t_0);u^{t_0}) \in \mathbf H$. There is a compact global attractor $\mathscr U \subset \mathbf H$ \cite{delay}, and we denote by $\mathscr B \subset \mathbf H$ a fixed absorbing ball for the dynamics of diameter $R$, such that $\mathscr U \subset \mathscr B$. We note that for $k>k_*>0$ the size of the absorbing ball $\mathscr B$ can be taken uniform in $k$.  In what follows, we can reduce the above delay plate dynamics $(T_t,\mathbf H)$ to those on the set $\mathscr B$ by waiting for absorption; this is to say that for any initial data ~~$||(u(t_0),u_t(t_0);u^{t_0})||_{\mathbf H} \le P$~~ there is a time of absorption $t_P$ such that $T_{t_P}(B_{\mathbf H}(P)) \subset \mathscr B$. Thus, by waiting sufficiently long (i.e., choosing $t_0 > t_P$---valid, as we are considering strong stability as $t \to +\infty$), we can reduce dynamics corresponding to any initial data to dynamics with initial data chosen from $\mathscr B$. We will denote this dependence by $R$---the diameter of $\mathscr B$.

Owing to the fact that  $q^{\cdot}$ is linear, we will  apply the following decomposition of  the solution (note this is a different decomposition than typically used 
for the proof of exponential attraction i.e.,  \cite{fgmz}). 

Take $(u(t_0),u_t(t_0);u^{t_0}) \in \mathscr B \subset \mathbf H$. We write $u = z+w$, where
 $z$ and $w$ correspond to the systems:
\begin{equation}\label{exp}\begin{cases} z_{tt} + (k+1) z_t + \beta z+ \Delta^2 z +\cF(z,u)  = \ q^z-Uz_x~~\text{ in } \Omega\\[.2cm]
\cF(z,u) =  [b-||\nabla u||^2]\Delta z\\[.2cm]
z = \Dn z =0 ~~\text{ on } \Gamma,\\
z(t_0) = u(t_0),~~ z_t(t_0) = u_t(t_0),~z|_{(t_0-t^*,t_0)}=u^{t_0}.\end{cases}\end{equation}
\begin{equation}\label{smooth}\begin{cases} w_{tt} + (k+1) w_t + \Delta^2 w  = p_0+ q^w-Uw_x- \cF(w,u)+\beta z~~ \text{ in } \Omega\\[.2cm]
 \cF(w,u) = [b-||\nabla u||^2]\Delta w\\[.2cm]
w=\Dn w=0~~\text{ on } \Gamma,\\ w(t_0) = 0, ~~w_t(t_0) = 0,~~w|_{(t_0-t^*,t_0)}=0.\end{cases}\end{equation}
\begin{remark}
Since we are taking null initial data for the the $w$ portion of the decomposition, we note that the presence of $z$ (via $u=z+w$) in the term $\mathcal F(w,u)$ is the driver of the dynamics. 
\end{remark}
 
Given $u$, the $z/w$ system is well-posed with respect to finite energy topology $\mathcal H$. Indeed, in light of the stability estimates shown below for the $z$ dynamics (with $u$ {\em given})  the analysis in \cite{delay} applies; from there, the $w$ dynamics are well-posed as a difference $w=u-z$. We state this as a lemma:
\begin{lemma}
Consider the dynamics in \eqref{exp} at time $t_0$, with initial data $(u(t_0),u_t(t_0);u^{t_0}) \in \mathscr B \subset \mathbf H$. Given $u \in C([t_0-t^*,T];H_0^2(\Omega))$ the dynamics corresponding to \eqref{exp} are well-posed in the sense of \cite{delay}. Then, considering $w=u-z$, with  $u$ corresponding to the solution in \eqref{reducedplate} on $\mathbf H$, the problem \eqref{smooth} is well posed on $\mathbf H$ with $(w,w_t;w^t)$ as a solution.
\end{lemma}
\noindent Let $$E_{\beta}(z(t))\equiv \dfrac{1}{2}[||\Delta z||^2+||z_t||^2+\beta ||z||^2].$$ For the decomposed plate dynamics, we will show the following supporting lemmas:
 \begin{lemma}\label{expattract}
There exists  $k_e(R)>0$ and  $\beta_e(R)>0$ such that for $k>k_e$ and $\beta>\beta_e$, the quantity $E_{\beta}(z(t))$ decays exponentially to zero. In fact, $T^z_t(x_0)=(z,z_t;z^t)$ decays uniformly exponentially to zero in $\mathbf H$ for $x_0 \in B_{\mathbf H}(R)$ (with rate depending on $k$ and $\beta$).
 \end{lemma}

  \begin{lemma}\label{regattract} There exists a $k_Q(R)$ such that for any $k>k_Q$ all $\beta >0$ the evolution $(w,w_t;w^{t})$ on $\mathbf H$ corresponding to \eqref{smooth} has that  $(w,w_t,w_{tt}) \in  C([t_0,\infty);H ^r),$ where $\ds H ^r \equiv 
   H^{4}(\Omega)\times H^{2}(\Omega)\times L_2(\Omega).$
 \end{lemma}
 
We will need to utilize the following estimates on $q$, coming from \cite{delay}. It is shown that they are valid for $q^{\cdot}$ as given in \eqref{reducedplate}---see also \eqref{qderiv} above. These estimates indicate a mechanism for {\it compensated compactness} acquired by the flow and translated into delay part of the plate system. 
\begin{align}\label{qests}
||q^u(t)|| \le Ct^*\int_{t-t^*}^t||\Delta u(\tau)|| d\tau,&\hskip1cm
||q^u(t)||_{-1} \le Ct^*\int_{t-t^*}^t||u(\tau)||_1 d\tau,\\
|\langle q^u_t(t),\psi\rangle | \le C\Big\{ ||u(t)||_1&~+||u(t-t^*)||_1+\int_{t-t^*}^t||u(\tau)||_2d\tau\Big\}||\psi||_1\\
||q^u_t(t)||\le  Ct^* \Big[
||\Delta u(t)||+ &||\Delta u(t-t^*)||+\int_{t-t^*}^t||u(s)||_3 ds\Big].\label{qests*}
\end{align}
 The last inequality follows from (\ref{qderiv}), after noting that $M_{\theta}^2$ is the second order differential operator controlled by $||\Delta u ||$ (due to zero Dirichlet boundary conditions imposed on the structure). 
Additionally, we note the following elementary bound:
$$\int_0^{t^*}\int_{t-s}^t G(u(\tau)) d\tau ds \le t^*\int_{t-t^*}^t G(u(\tau)) d\tau,$$ where $G$ is any positive functional.

  \subsection{Stability of \eqref{exp}---Proof of Lemma \ref{expattract}}\label{expsec}
 
We consider the Lyapunov approach which was used in \cite{springer,delay} (based upon \cite{Memoires}) to show the existence of an absorbing set (and  used in \cite{conequil1} to show Hadamard continuity of the full flow-plate dynamics). 
Define the Lyapunov-type function:
\begin{align}
V(T^z_t(x_0)) \equiv &~E_{\beta}(z)-\langle q(z^t),z\rangle+\langle z_t,z\rangle +\frac{k}{2}||z||^2\\\nonumber
&+\mu\Big( \int_{t-t^*}^t||\Delta z||^2ds+\int_0^{t^*}\int_{t-s}^t ||\Delta z(\tau)||^2 d\tau ds\Big),
\end{align}
where $T^z_t(x_0) = (z(t),z_t(t);z^t)$ for $t \ge t_0,$ and $\mu$ is some positive number to be specified below. Using Young's inequality and \eqref{qests}, we have immediately that
\begin{equation}\label{energybounds}
c(\mu)E_{\beta}(z(t)) \le V(T^z_t(x_0)) \le C(k)E_{\beta}(z(t))+C(\mu,t^*)\int_{t-t^*}^t ||(z(\tau))||_2^2d\tau
\end{equation}
The constant $c(\mu)$ does not depend on either damping parameter.
We now continue with the Lyapunov calculus by computing the time derivative of $V(T^z_t(x_0))$. Since the dynamics has been shown sufficiently regular for regular data, we can consider such regular initial data and perform differential calculus on a priori smooth solutions. Density and passage with the limit on initial data will allow to draw the same conclusion for finite energy initial data. 
\begin{align}
\dfrac{d}{dt} V(T^z_t(x_0))=
&\nonumber~\dfrac{d}{dt}E_{\beta}(z(t))+\langle z_{tt},z\rangle +||z_t||^2+k \langle z_t,z\rangle \\\nonumber
&-\langle q,z_t\rangle-\langle q_t,z\rangle \\\nonumber
&+\mu \dfrac{d}{dt} \Big\{\int_{t-t^*}^t||\Delta z||^2ds+\int_0^{t^*}\int_{t-s}^t ||\Delta z(\tau)||^2 d\tau ds \Big\}\\[.1cm]
=&~-k||z_t||^2+(\mu t^*+\mu-1)||\Delta z||^2-\beta||z||^2 \\\nonumber
&- \mu||\Delta z(t-t^*)||^2-\mu\int_{t-t^*}^t||\Delta z||^2 ds\\\nonumber
&+ \langle q,z\rangle-\langle q_t,z\rangle- U\langle z_x,z \rangle-U\langle z_x,z_t \rangle\\\nonumber
&-\langle\cF(z,u),z_t\rangle-\langle \cF(z,u),z \rangle
\end{align}
We twice made use of relation \eqref{exp}: $\ds z_{tt}+\beta z+(k+1)z_t+\Delta^2z+\cF(z,u)=q^z-Uz_x.$
We estimate from above; first we note that
\begin{align}
||\cF(z,u)||=& ||~[b-||\nabla u||^2]\Delta z~|| \le C(R,b)||z||_2.
\end{align}
Then, using \eqref{qests}, Young's inequality, and compactness, we have:
\begin{align*}
|\langle\cF(z,u),z_t\rangle|+|\langle \cF(z,u),z \rangle|  \le &~\epsilon ||\Delta z||^2+C_{\epsilon}(R,b)[||z_t||^2+||z||^2]\\[.2cm]
|\langle q^z(t),z\rangle|  \le &~\epsilon \int_{t-t^*}^t ||\Delta z||^2 +C_{\epsilon}(t^*)||z||^2 \\[.2cm]
|\langle q^z_t(t),z\rangle | \le &~\epsilon\left[||\Delta z(t)||^2+||\Delta z(t-t^*)||^2+\int_{t-t^*}^t||\Delta z(\tau)||d\tau\right]\\\nonumber&+C_{\epsilon}(t^*)||z||^2\\[.2cm]
U|\langle z_x,z \rangle |+U|\langle z_x,z_t \rangle | \le &~\epsilon [||\Delta z||^2+||z_t||^2]+C_{\epsilon}(U)||z||^2 
\end{align*}

We may choose $\mu$ sufficiently small, then $\epsilon$ small, then $\beta$ (depending on $R$ and $\mu$) sufficiently large, and finally $k$ sufficiently large (depending on $\beta$ and $R$).

This yields:
\begin{lemma}\label{le:48}
There exists $k_e$ and $\beta_e$ so that for $k>k_e>0$ and $\beta>\beta_e>0$:
\begin{equation}\label{goodneg}
\dfrac{d}{dt}V(T^z_t(x))\le-c(k,\beta)\left\{||z_t||^2+||\Delta z||^2 +||z||^2+||\Delta z(t-t^*)||^2+\int_{t-t^*}^t||\Delta z(\tau)||^2d\tau\right\},
\end{equation}
where $c(k,\beta) \to \infty$ as $\min\{k,\beta\}\to \infty$.
\end{lemma}
From this lemma, and the upper bound in \eqref{energybounds}, we have for some $\gamma(k,\beta)>0$: 
\begin{equation*}
E_{\beta}(z(t))\le CV(T^z_t(x_0)) \le C V(x_0)e^{-\gamma t},~~t \ge t_0.
\end{equation*} 
This concludes the proof of the uniform exponential stability of the $z$ portion of the decomposed dynamics on $\mathscr B$.

\subsection{Regularity of \eqref{smooth}---Proof of Lemma \ref{regattract}}\label{regsec}
Again, recall that we have chosen $t_0$ large enough so that $(u(t_0),u_t(t_0);u^{t_0}) \in \mathscr B$. Let $T^w_t(\mathbf 0)=(w(t),w_t(t);w^t)$ correspond to the \eqref{smooth} dynamics.
First, we note: 
\begin{lemma}\label{heyo}
The dynamics $T_t^w(\mathbf 0)$ (such that $T_t^w(\mathbf 0)+T_t^z(x_0)=T_t(x_0)$, and $u=w+z$) are uniformly bounded on $\mathbf H$. This is to say that:
$$||(w(t),w_t(t);w^t)||_{\mathbf H} \le C(R), ~\forall~~t\ge t_0.$$\end{lemma}
\begin{proof}[Proof of Lemma \ref{heyo}]
This follows immediately from the existence of a uniform absorbing set $\mathscr B$ for the dynamics $T_t(x_0) = (u(t),u_t(t);u^t)$ on $\mathbf H$, and the uniform exponential stability of the $T^z_t(x_0)$ dynamics. \end{proof}

To continue with the proof of Lemma \eqref{regattract}, recall $\cF(w,u) = [b-||\nabla u||^2]\Delta w$. Consider the time-differentiated $w$ dynamics; let $\overline w = w_t$:
\begin{equation}\label{whatthe}\begin{cases} \overline w_{tt} + (k+1) \overline w_t + \Delta^2 \overline w  =  \dfrac{d}{dt}\left\{q(w^t)\right\}-\dfrac{d}{dt}\left\{\cF(w,u)\right\}-U\overline w_x+\beta z_t~~ \text{ in } \Omega\\
\overline w=\Dn \overline w=0~~\text{ on } \Gamma\\
\overline w(t_0) = 0, ~~\overline w_t(t_0) = 0,~~\overline w|_{(t_0-t^*,t_0)}=0.\end{cases}\end{equation}

Recall, the term~ $||z_t||$ decays exponentially, and is thus bounded. We will now obtain a global in time bound on~$||\overline w||_2+||\overline w_t||$ in what follows. We now invoke the well-known {\em exponential decay of the linear, damped (static {and} viscous) plate equation}  for 
$$\overline w_{tt} + (k+1)\overline w_t + K \overline w + \Delta^2 \overline w =G.$$  (To note this decay one can, for instance, see the argument just given in Section \ref{expsec} leading up to Lemma \ref{le:48}.) We consider:
\begin{align} 
\overline w_{tt} + (k+1) \overline w_t+K\overline w + \Delta^2 \overline w=&~  \dfrac{d}{dt}\left\{q^w\right\}-\dfrac{d}{dt}\left\{\cF(w,u)\right\}-U\overline w_x+\beta z_t+K\overline w\\=&~  G(u,z,w,\overline w, w^t).
\end{align}  Let $\mathbf w(t) = (\overline w(t),\overline w_t(t))$. For $t_0$ (again, sufficiently large) we utilize the {\em variation of parameters formula}:
\begin{align}\label{varpar0}
\Big|\Big|\mathbf w(t)\Big|\Big|_{Y_{pl}}  \le &~
C \Big|\Big|\mathbf w(t_0)\big)\Big|\Big|_{Y_{pl}}+\int_{t_0}^t e^{-\omega(k,K)(t-s)}\Big[||q_t^w ||_0+\left|\left|\dfrac{d}{dt}\cF(w,u)\right|\right|_0\\\nonumber&+U||\overline w||_1+K||\overline w||_0+\beta||z_t||_0\Big]ds.
\end{align}
Noting that ~$\omega \to \infty$ ~as ~$\min [k,K] \to \infty$~ (see the analysis  \eqref{goodneg}). The critical terms involve $q_t^w$ and $\dfrac{d}{dt}\mathcal F(w,u)$, as the other terms are bounded by Lemma \ref{heyo}.

For the term $\partial_t\left[q^w\right]$ we note the bounds in \eqref{qests*} above (and \cite{delay} for more details):
\begin{align}||q^w_t(\tau)||\le&~  Ct^* \Big[
||\Delta w(\tau)||+ ||\Delta w(\tau-t^*)||+\int_{\tau-t^*}^t||w(s)||_3 ds\Big] \end{align}
At this point, elliptic theory can be applied: a bound on $\|w\|_4$ is to be calculated from an elliptic equation in terms of ~$\|(\overline w,\overline w_t) \|_{{Y}_{pl}}$.
Indeed, we consider biharmonic problem with  the clamped boundary conditions.
$$ \Delta^2 w = - w_{tt} - (k+1)w_t  -\mathcal F(w,u) -Uw_{x} - q^{w}+\beta z;~~~~w = \Dn w =0 ~\text{on}~ \partial \Omega.$$
This gives, via elliptic estimates (after accounting for  the a priori bound in Lemma \ref{heyo}):
\begin{align*}\|w\|_4 \leq& ~C  \|w_{tt} - (k+1) w_t  -\mathcal F(w,u) -w_{x}  - q^{w}+\beta z\|
\\  \leq& ~C \big[ ||w_{tt}|| + (k+1)  ||w_t|| +(1+||u||^2_1)||w||_2 + U ||w||_1 + (1+t^*)\sup_{ s\in[t-t^*,t]} ||w(s)||_2   +\beta C \big]\\
\leq&~ C   [||(w_t,w_{tt})||_{Y_{pl}} +  1  ] .
\end{align*} 
Thus we have:
\begin{equation}\label{w4}
\|w(t) \|_4 \leq  C(R)\left[\|(\overline w, \overline w_{t})\|_{Y_{pl}} +  1\right],
\end{equation}

 We also have:
\begin{align}\label{timenon}
\Big|\Big|\dfrac{d}{dt}\cF(w,u) \Big|\Big|_0
=~ \Big|\Big|(\nabla u, \nabla u_t)\Delta w+||\nabla u||^2\Delta w_t \Big|\Big|_0
\le~\left|\left|(\Delta u,u_t)\Delta w\right|\right|+||\nabla u||^2\left|\left|\Delta \overline w\right|\right|.
\end{align}
Thus: \begin{equation}\label{keyvK} \left|\left| \dfrac{d}{dt}\cF(w,u)\right| \right|_0 \le C(R)\big(1+||\Delta \overline w||\big).\end{equation}
 Returning to \eqref{varpar0}, and implementing the above bounds in \eqref{w4} and \eqref{keyvK}, we have
\begin{align}
\Big|\Big|\big(\overline w(t),\overline w_t(t)\big)\Big|\Big|_{Y_{pl}} \le \nonumber&~
C ||(\overline w(t_0),\overline w_t(t_0))||_{Y_{pl}}\\&+C(R,U)\int_{t_0}^t e^{-\omega(k,K)(t-s)}\Big[1+K+||\overline w||_2+\int_{s-t^*}^s ||w(\tau)||_3d\tau\Big] ds \\
\le \nonumber&~C ||(\overline w(t_0),\overline w_t(t_0))||_{Y_{pl}}\\&+C(R,U)\int_{t_0}^t e^{-\omega(k,K)(t-s)}\Big[1+K+||\overline w||_2+\int_{s-t^*}^s ||(\overline w,\overline w_t)||_{Y_{pl}}d\tau\Big] ds 
\end{align}
At this point we note by our choice of decomposition---the $w$ problem with null initial data---the initial condition $\overline w \big|_{(t_0-t^*,t^*)}=0$ (and hence $\overline w_t =0$ as well) on $(t_0-t^*,t_0)$. Thus:
\begin{align}
\Big|\Big|\big(\overline w(t),\overline w_t(t)\big)\Big|\Big|_{Y_{pl}}
\le &~
C ||(\overline w(t_0),\overline w_t(t_0))||_{Y_{pl}}\\\nonumber&+C(R,U,t^*)\big[1+K+\sup_{s \in [t_0,t]}||(\overline w,\overline w_t)||_{Y_{pl}}\big]\int_{t_0}^t e^{-\omega(k,K)(t-s)}ds \\
\le &~C\Big|\Big|\big(\overline w(t_0),\overline w_t(t_0)\big)\Big|\Big|_{Y_{pl}}\\\nonumber&+C(R,U,t^*)\big[1+K+\sup_{s \in [t_0,t]}||(\overline w,\overline w_t)||_{Y_{pl}}\big]\dfrac{1-e^{-\omega(k,K)(t-t_0)}}{\omega(k,K)}.
\end{align}
A global-in-time bound on $(w_t,w_{tt})=(\overline w, \overline w_t)$ in $Y_{pl}$ follows by (i) taking supremums in $t \in [t_0,\infty)$, then (ii) (possibly) up-scaling ~~$\min[k,K]$ ~sufficiently (and thus scaling $\omega(k,K)$) so as to absorb the supremum on the RHS.  This completes the proof of Lemma \ref{regattract}.

\subsection{Decomposition of Flow Dynamics}\label{flowdecomp}
Consider data $y_0 \in Y_{\rho}$. At a time $t_0$ sufficiently large, we employ the reduction result Theorem \ref{rewrite} and utilize initial data $x_0=(u(t_0),u_t(t_0);u^{t_0})$ coming from the full flow plate dynamics, and denote the resulting semiflow from \eqref{reducedplate} by $T_t(x_0)$.

We consider the decomposition of the dynamics $T_t(x_0)=(u(t),u_t(t);u^t)$ into $z$ and $w$ components as in \eqref{exp} and \eqref{smooth}.
We then consider the potential flow equation, with $w$ given from \eqref{smooth} and initial time $t_0$ sufficiently large. 
\begin{equation}\label{floweq1}
\begin{cases}
(\partial_t+U\partial_x)^2\phi^w=\Delta \phi^w & \text{ in }~\mathbb R_+^3\\
\partial_{\nu} \phi^w = -[w_t+Uw_x]_{ext} & \text{ in }~\mathbb R^2\\
\phi^w(t_0)=0;~~\phi^w_t(t_0)=0.
\end{cases}
\end{equation}
In this case, via the smoothing in Lemma  \ref{regattract}, the Neumann data~ $[w_t+Uw_x]_{\text{ext}} \in H^1(\Omega) $~ is sufficiently regular for Theorem \ref{flowpot} to apply. This yields a resulting potential flow $(\phi^w,\phi_t^w) \in  C(0, \infty; H^1(\realsthree_+) \times L_2(\realsthree_+))$. Owing to the calculations in \cite{springer} and \cite[p. 241]{b-c-1} (due to the zero flow initial condition), we have that:
\begin{equation}\label{flowdecomp1}
r_{\Omega}tr\big[[\partial_t+U\partial_x]\phi^w\big] =-(\partial_t+U\partial_x)w-q(w^t).
\end{equation}
With this identity, we can consider the {\em non-autonomous} semiflow $S_t^w(\cdot)$ on $Y_{\rho}$, in reference to the flow-plate problem:
\begin{equation}\label{wfull}\begin{cases}
w_{tt} + k w_t + \Delta^2 w  = p_0+(\partial_t+U\partial_x)\phi^w\big|_{\Omega}-\cF(w,u)+\beta z~~ &\text{ in } \Omega\\
w=\Dn w=0~~\hskip1cm\text{ on } \Gamma\\
w(t_0) = 0, ~~w_t(t_0) = 0\\
(\partial_t+U\partial_x)^2\phi^w=\Delta \phi^w \hskip1cm \text { in } \realsthree_+ \times (0,T),\\
\phi^w(t_0)=0;~~\phi^w_t(t_0)=0,\\
\partial_{\nu} \phi^w = -\big[(\partial_t+U\partial_x)w\big]_{ext}\hskip1cm\text{ on } \realstwo_{(x,y)} \times (0,T).
\end{cases}
\end{equation}
 Note that $\cF (w,u) $ is linear in $w$, and we have the estimate 
 \begin{equation}\label{F}
 ||\cF(w,u) ||\leq C ||\nabla u||^2_{L^2(\Omega) } ||w||_2 \leq C ||u||^2_2 ||w||_2.
 \end{equation}
Owing to the well-posedness of the full flow-plate problem \eqref{flowplate} \cite{webster} and the estimate in (\ref{F}),  one can assert well-posedness  in $Y$  of the $(\phi^w,w)$---linear---dynamics by  appealing 
to the results for the {\em linear} flow in \cite[Theorem 2.1, pp. 3128--3128]{webster} (see \cite{jadea12} as well).  We now  consider the restriction to the flow $\phi$ from the full $S_t(y_0)$. By considering the difference $\phi-\phi^w$, we can construct the function $\phi^z$, satisfying the corresponding potential flow equation with Neumann data $$\partial_{\nu}\phi^z=-[z_t+Uz_x]_{\text{ext}}$$ on $\mathbb R^2$. Combining this potential flow $\phi^z$ and the delayed $z$ plate dynamics, as well as considering the difference between the flow-plate systems \eqref{flowplate} and \eqref{wfull}, we have constructed a {\it non-autonomous }  flow $S^z_t(\cdot)$ on $Y_{\rho}$ corresponding to the dynamics:
\begin{equation}\label{zfull}\begin{cases}
z_{tt} + k z_t + \Delta^2 z  = (\partial_t+U\partial_x)\phi^z\big|_{\Omega}-\cF(z,u)-\beta z~~ &\text{ in } \Omega\\
z=\Dn z=0~~\hskip1cm\text{ on } \Gamma\\
z(t_0) = u(t_0), ~~z_t(t_0) = u_t(t_0)\\
(\partial_t+U\partial_x)^2\phi^z=\Delta \phi^z \hskip1cm \text { in } \realsthree_+ \times (0,T),\\
\phi^z(t_0)=\phi(t_0);~~\phi^z_t(t_0)=\phi_t(t_0),\\
\partial_{\nu} \phi^z = -\big[(\partial_t+U\partial_x)z \big]_{\text{ext}}  \hskip1cm\text{ on } \realstwo_{(x,y)} \times (0,T).
\end{cases}
\end{equation}
Note: by considering $k$ and $\beta$ sufficiently large (as in Lemma \ref{expattract}), we have that $(z,z_t;z^t)$ decays exponentially in $Y$  (uniformity with respect to $R$). 

 Recall that $\mathcal N$ corresponds to the points $(u,\phi)$ which satisfy \eqref{static}. Consider the stationary problem {\em for some given function ~$u\in H_0^2(\Omega)$}:
\begin{equation}\label{staticz}
\begin{cases}
\Delta^2z+\beta z+\mathcal F(z,u)=Ur_{\Omega}tr[\partial_x \phi^z]& \xb \in \Omega\\
z=\Dn z= 0 & \xb \in \Gamma\\
\Delta \phi^z -U^2 \partial_x^2\phi^z=0 & \xb \in \realsthree_+\\
\Dz \phi^z = U\partial_x z_{\text{ext}}  & \xb \in \partial \realsthree_+
\end{cases}
\end{equation}
Let $$\mathcal N^z \equiv \{(\widehat z,\widehat{\phi^z}) \in H_0^2(\Omega) \times W^1(\realsthree_+): (\widehat z,\widehat{\phi^z}) ~\text{satisfy \eqref{staticz} variationally}\},$$  and let $\mathcal N^{z,*} \subset \mathcal N^z$ be the subset of points of the form $\{0; \widehat{\phi^z} \}$.  

\subsection{ $  (z,\phi^z) $ Dynamics }\label{zsec}
In this section we will show that the desired convergence to equilibria result, as in Theorem \ref{conequil}, holds for the $(z,\phi^z)$ dynamics. This will be accomplished by critically using the uniform exponential decay of the $(z,z_t)$ plate dynamics, and running an approximation argument with smooth data as in \cite{conequil1}. The main result of this section is: 
\begin{theorem}\label{dichotoo1} 
Assume $k>k_e$ and $\beta > \beta_e$, with $y_0 \in Y_{\rho_0}$; then for any $t_n \to +\infty$ there is a subsequence $t_{n_j}$ and a point $\widehat{y^z}=(\widehat z,0;\widehat{\phi^z},0)$ with $(\widehat z,\widehat{\phi^z}) \in \mathcal N^{z,*}$ (i.e., $\widehat z =0$) so that or any $\rho>0$.

$$\ds \lim_{j\to\infty} ||S^z_{t_{n_j}}(y_0)-\widehat{y^z}||_{Y_{\rho}}=0,$$ 
\end{theorem}
\begin{proof}[Proof of Theorem \ref{dichotoo1}]  The proof follows through several lemmas. Note that the linear part of \eqref{zfull} is equivalent to the system \eqref{withbeta} (since both have static damping, $\beta>0$). The decomposed nonlinear term $\mathcal F(z,u)$  enjoys the estimate. 
\begin{align}\label{unbound}
||\cF(z,u)||=&  \le C(R,b)||z||_2.
\end{align}  
\begin{lemma}\label{hadcont*}[Uniform-in-Time Hadamard Continuity  for $S_t^z(\cdot) $]
Given initial datum $y_0^m=(u^m_0,u^m_1;\phi^m_0,\phi^m_1)$ and $y_0=(u_0,u_1;\phi_0,\phi_1)$, with corresponding decompositions as in \eqref{zfull} and \eqref{wfull}, we have that 
$$\sup_{t \ge 0} ||S_t^z(y_0^m) - S_t^z(y_0)||_{Y_{\rho}} \to 0~ \mbox{ as} ~ y_0^m \to y_0,  in Y$$ 
\end{lemma}

\begin{proof}[Proof of Lemma \ref{hadcont*}]
With a minor modification to the proof of Theorem \ref{strongcont} in \cite{conequil1} ($f(z)$ replaced by $\mathcal F(z,u)$), we can obtain global-in-time Hadamard continuity. Indeed, consider the differences $\mathbf z=z^m-z$ and $\Phi^z = \phi^{z^m}-\phi^z$ which satisfy the system:
\begin{equation}\label{zfulldiff}\begin{cases}
\mathbf z_{tt} + k \mathbf z_t + \Delta^2 \mathbf z  = (\partial_t+U\partial_x) \Phi^z\big|_{\Omega}-[\cF(z^m,u^m)-\cF(z,u)]-\beta \mathbf z&\text{in } \Omega\\
z=\Dn z=0~~\hskip1cm\text{ on } \Gamma\\
\mathbf z(t_0) = u^m(t_0)-u(t_0), ~~\mathbf z_t(t_0) = u^m_t(t_0)-u_t(t_0)\\
(\partial_t+U\partial_x)^2 \Phi^z=\Delta  \Phi^z \hskip1cm \text { in } \realsthree_+ \times (0,T),\\
 \Phi^z(t_0)=\phi^m(t_0)-\phi(t_0);~~ \Phi^z_t(t_0)=\phi^m_t(t_0)-\phi_t(t_0),\\
\partial_{\nu} \phi^z = -\big[(\partial_t+U\partial_x)\mathbf z \big]_{\text{ext}}  \hskip1cm\text{ on } \realstwo_{(x,y)} \times (0,T).
\end{cases}
\end{equation}
We then utilize the velocity multipliers $\mathbf z_t$ and $ \Phi_t$ (respectively). The linear portion of the dynamics yields the standard energy identity for the differences:
  \begin{align}\label{needed}
 E_z(t) + \int_0^tk\|\mathbf z_t\|^2 d\tau = &~E_z(0)+2U\lb  \Phi(0),\mathbf z_x(0)\rb-2U\lb  \Phi(t),\mathbf z_x(t)\rb\notag\\&+\int_0^t \left\lb \cF(z^m,u^m)-\cF(z,u),\mathbf z_t\right\rb_{\Omega} d\tau,  \end{align}
 where we have utilized the notation, \begin{align*}
 E_z=&~ \frac{1}{2}\left(\|\Delta \mathbf z\|^2+\|\mathbf z_t\|^2+\beta\|\mathbf z\|^2+\|\nabla  \Phi\|^2 -U^2\| \Phi_x\|^2+\| \Phi_t\|^2\right).
 \end{align*} 
 
 By the bounds on $E_{int}$ as in Lemma \ref{energybound}, as well as the global-in-time energy bounds (Lemma \ref{globalbound}),
 one obtains
 \begin{align}
 E_z(t) + \int_0^t k\|\bz_t\|^2d\tau & \nonumber \\ \le C(U)E_z(0)+&C\|\bz_x(t)\|^2+C\int_0^t \left|\left\lb \cF(z^m,u^m)-\cF(z,u),\mathbf z_t\right\rb_{\Omega} \right|d\tau.
 \end{align} Via compactness ~~$\|\bz_x\|^2 \le \epsilon \|\Delta \bz \|^2 + C(\epsilon) \|\bz\|^2$, and hence
 \begin{align}
 E_z(t) + \int_0^t k\|\bz_t\|^2d\tau & \nonumber \\ \le C(U)E_z(0)+&C\|\bz(t)\|^2+C\int_0^t \left|\left\lb \cF(z^m,u^m)-\cF(z,u),\mathbf z_t\right\rb_{\Omega} \right|d\tau
\end{align}
 Observing that \begin{equation}\label{oneusing} \|\bz(t)\|^2 = \left|\left|\int_0^t \bz_t(\tau)d\tau +\bz(0) \right|\right|^2 \le \int_0^t \|\bz_t\|^2d\tau +\|\bz(0)\|^2 \le \int_0^t E_z(\tau)d\tau+E_z(0),\end{equation}
 we then have:
 \begin{align}\label{okayy1}
 E_z(t) + \int_0^t k\|\bz_t\|^2d\tau \le C(U)E_z(0)+C(R)\int_0^t E_z(\tau)d\tau+C\int_0^t \left|\left\lb \cF(z^m,u^m)-\cF(z,u),\mathbf z_t\right\rb_{\Omega} \right|d\tau
 \end{align}
  (where $C(R)$ denotes the dependence of the constant on the size of the ball containing the initial datum $y^0$ and $y^0_m$).
 To note {\em finite-in-time} Hadamard continuity, let $T>0$ be fixed. From \eqref{needed}, we need only consider the nonlinear term:
 $$\int_0^T \big|\langle \mathcal F(z^m,u^m)-\mathcal F(z,u),\mathbf z_t \rangle\big| d\tau.$$
 Note the identity:
 \begin{align}
(b-||\nabla u^m||^2)\Delta z^m-(b-||\nabla u||^2)\Delta z =\nonumber&\\b\Delta(z^m-z)&-(||\nabla u^m||^2-||\nabla u||^2)\Delta z^m-||\nabla u||^2\Delta(z^m-z)
 \end{align}
 And thus:
 \begin{equation}\label{liplike}||\mathcal F(z^m,u^m)-\mathcal F(z,u)|| \le C(b,R)\big[||u^m-u||_2+||z^m-z||_2\big].\end{equation}
 Implementing this bound, we see that
 \begin{align}
\Big| \int_0^T \langle \mathcal F(z^m,u^m)-\mathcal F(z,u),\mathbf z_t \rangle d\tau\Big| \le &~C(b,R)\int_0^T [||\mathbf u||_2+||\mathbf z||_2]||\mathbf z_t|| d\tau\\
\le &~C_{\epsilon}(b,R)\big[\int_0^TE_z(\tau)d\tau+\int_0^T ||u^m-u||_2^2 d\tau\big]\label{okayy2}.
 \end{align}
 We can now invoke the Hadamard continuity of $S_t(\cdot)$ for the $u$-dynamics on any finite $[0,T]$ (Theorem \ref{hadcont}), as well as the initial conditions in \eqref{zfulldiff}, and thus
 $$||u^m(\tau)-u(\tau)||^2_2 \le C(T)E_z(0).$$
Gronwall's inequality applied in \eqref{okayy1} with \eqref{okayy2} then yields that
 \begin{equation}\label{gronwall} E_z(t) \le C(b,U,R,T)E_z(0). \end{equation}
Hence, for any {\em fixed} $T$, we obtain a Hadamard continuity of the semiflow.

We now address this continuity on the infinite-time horizon. We return to \eqref{needed}:
\begin{align}
E_z(t)+\int_{0}^t k\|\bz_t\|^2 d\tau \le & ~ C\Big\{E_z(0)+\|\bz(t)\|^2_1+\Big|\int_{0}^t\langle \cF(z^m,u^m)-\cF(z,u),\bz_t\rangle d\tau\Big| \Big\}\notag\\
\le C\Big\{E_z(0)+\epsilon E_z(t)&~ +\|\bz (t)\|^2+\int_{0}^t \| \cF(z^m,u^m)-\cF(z,u)\|_0\| \bu_t\|_0 d\tau \Big\}
\end{align}
Again using \eqref{oneusing}, we obtain:
\begin{align}
E_z(t)+\int_{0}^t k\|\bz_t\|^2 d\tau \le & ~ C\Big\{E_z(0)+\int_0^t\|\bz_t\|^2 d\tau+\int_{0}^t \| \cF(z^m,u^m)-\cF(z,u)\|_0\| \bu_t\|_0 d\tau \Big\}
\end{align}
We invoke the bound \eqref{liplike} and note the global-in-time bounds on the $(u^m,u^m_t)$, $(u,u_t)$ and $(z^m ,z^m_t)$, and $(z,z_t)$ dynamics in $Y_{pl}$:
$$||\mathcal F(z^m,u^m)-\mathcal F(z,u)|| \le C(b,R)[||u^m-u||_2+||z^m-z||_2]\le C(R),$$ and thus
\begin{align*}E_z(t)+\int_{0}^t k\|\bz_t\|^2 d\tau \le & ~ C\Big\{E_z(0)+\int_{0}^t\big(\|\bz_t\|^2+C(R)\|\bz_t\| \big)d\tau \Big\}\end{align*}
Scaling up $\beta$ and $k$ provides exponential decay of $z_t$ and $z^m_t$ to $0$, and thus $\mathbf z_t \to 0$ exponentially---uniform with respect to $y_0^m,y_0 \in B_{Y}(R)$ (see above Lemma \ref{expattract}). 
 Via this uniform exponential decay, for every $\epsilon>0$ there exist a time $T^*>t^{\#}$ such that\footnote{$T^*$ depends on the underlying parameters of the problem $p_0,\Omega,U$, as well as the time of absorption for $\mathscr B$, and the size of the support of the initial flow data $\rho$.} $$\int_{T^*}^{\infty}\big(\|\bz_t\|^2+C(R)\|\bz_t\|\big) d\tau \le \epsilon.$$

 Thus, for any $t>T^*$  we may write:
 \begin{align}E_z(t) \le &   C\Big\{E_z(0)\!+\!\int_{0}^{T^*}\big(\|\bz_t\|^2\!+\!C(R)\|\bz_t\|\big) d\tau\!+\!\int_{T^*}^{\infty}\big(\|\bz_t\|^2\!+\!C(R)\|\bz_t\|\big) d\tau \Big\} \end{align}
 Utilizing \eqref{gronwall}, for any $\epsilon>0$ we have:
 \begin{align}
 E_z(t) \le C(R,T^*)\big(E_z(0)+E^{1/2}_z(0)\big)+\epsilon/2
 \end{align}
 Taking $y^0_m$ sufficiently close to $y^0$ will yield that ~$C(R,T^*)\big(E_z(0)+E_z^{1/2}(0)\big) \le \epsilon /2$. This concludes the proof of Lemma \ref{hadcont*}.
\end{proof}
With uniform-in-time Hadamard continuity available, to continue with an {\em approximation by smooth data} approach for the $(z,\phi^z)$ dynamics, we must show that the desired convergence to equilibria holds for {\em smooth data} in the $(z,\phi^z)$ equation. 
\begin{lemma}\label{zalso}
Assume that $$x_0=(u(t_0),u_t(t_0);u^{t_0}) \in (H^4\cap H_0^2)(\Omega)\times H_0^2(\Omega) \times L_2\left(-t^*,0;(H^4\cap H_0^2)(\Omega)\right).$$ Then the dynamics $T_t^z(x_0)$ (such that $T_t^w(\mathbf 0)+T_t^z(x_0)=T_t(x_0)$, and $u=w+z$) are uniformly bounded on $(H^4\cap H_0^2)(\Omega)\times H_0^2(\Omega) \times L_2\left(-t^*,0;(H^4\cap H_0^2)(\Omega)\right)$. This is to say that:
$$||(z(t),z_t(t);z^t)||_{(H^4\cap H_0^2)(\Omega)\times H_0^2(\Omega) \times L_2\left(-t^*,0;(H^4\cap H_0^2)(\Omega)\right)} \le C(R), ~\forall~~t\ge t_0.$$ \end{lemma}
\begin{proof}[Proof of Lemma \ref{zalso}] We do not repeat the argument in full, as it is directly follows the proof of Theorem \ref{highernorms}, mutatis mutandis. Indeed, we consider smooth initial (as above), differentiate the equation \eqref{exp}, and perform the Lyapunov analysis in the proof of Theorem \eqref{highernorms}. The key fact (which greatly simplifies the proof in this case) is below:
\begin{align}\dfrac{d}{dt}\cF(z,u) = -2(\nabla u,\nabla u_t)\Delta z&+[b-||\nabla u||^2]\Delta z_t=2(\Delta u,u_t)\Delta z+[b-||\nabla u||^2]\Delta z_t\\[.4cm]
\left|\left|\dfrac{d}{dt}\cF(z,u) \right|\right|_0 \le &~C(b,R)[1+||z_t||_2].\label{thisthis} \end{align}
\begin{remark}\label{notkarman} Note this type of estimate {\em fails} for von Karman dynamics because of the precise structure of the nonlocality; this is precisely the stage at which our {\em decomposition} approach will not obtain for the analogous von Karman problem.\end{remark}
At this point, we repeat the arguments in the proof of the propagation of regularity result Theorem \ref{highernorms}. Specifically, the analysis of the term $B(t)$ (involving the decomposition involving terms $P_1,Q_1$) is not needed here. In particular, we take $P_1$ and $Q_1$ to be zero here. Indeed, utilizing \eqref{thisthis}, we have the estimate:
\begin{equation}\label{oneappealing}
\left|\left\langle \dfrac{d}{dt}\cF(z,u),z_{tt} \right\rangle\right| \le \epsilon||\Delta z_t||^2+C_{\epsilon}||z_{tt}||^2+C_{\epsilon}(R).
\end{equation}
At the junctures of \eqref{bee}, \eqref{w1}, and in the proof of Lemma \ref{lyapo}, we implement \eqref{oneappealing} and up-scale $k$ appropriately. 
\end{proof}

\begin{remark}
The subtlety needed in showing the propagation result in Theorem \ref{highernorms} (involving the decomposition into $P_1$ and $Q_1$) is needed to demonstrate that the {\em propagation of regularity for the original full flow-plate dynamics in $(u,\phi)$ can be obtained for ANY damping coefficient $k>0$}. In proving Lemma \ref{zalso} above, we already require $k$ large to yield other supporting results, and thus scaling $k$ up to simplify the argument for the reduced $z$ dynamics \eqref{exp} analogously is acceptable.\end{remark}

As a corollary, the global-in-time boundedness in Lemma \ref{zalso} yield the desired convergence to equilibria result for smooth initial data to the full $z$ problem \eqref{zfull}.
\begin{corollary}\label{zcor} Consider $y_0^M \in \mathscr D(\mathbb T) \cap Y$. 
 Thus, for any sequence $t_n \to \infty$, there is a subsequence {\em depending on $M$} (also labeled $t_n$) and a point $(\widehat z,\widehat{\phi^z}) \in \mathcal N^z$ (depending on the subsequence, and thus depending on $M$) such that~~
$$
\lim_{n \rightarrow \infty} ||S^z_{t_n}(y^M_0)-\widehat{y^z}||_{Y_{\rho}}=0,~~\rho>0.$$
\end{corollary}
\begin{proof}[Proof of Corollary \ref{zcor}]
Due to the identical structure of the $(z,\phi^z)$ dynamics in \eqref{exp} to those of \eqref{flowplate} (with $f(u)$ replaced by $\cF(z,u)$ and removing $p_0$), the proof of Theorem \ref{highernorms0} can be repeated identically for the $z$ dynamics \eqref{exp}, bearing in mind the bound in Lemma \ref{zalso}. (Recall: Theorem \ref{highernorms0}---which translates a global-in-time bound on plate dynamics into the desired strong convergence to equilibria for the full flow-plate dynamics---does not ``care" about the structure of the specific plate equation, only the nature of the coupling. See Remark \ref{dada}.)
\end{proof}
We are now in a position to finish the proof of Theorem \ref{dichotoo1}.
Now, let $\epsilon>0$ be given. We consider a sequence of smooth initial data $y_0^m \in \mathscr D(\mathbb T)$ with the property that $y_0^m \to y_0$ in $Y$. We may choose $M$ sufficiently large so that $$ \sup_{t > 0} \|S^z_{t}(y_0 ) -  S^z_{t} (y_0^{M}) \|_{Y_{\rho}}< \epsilon/2.$$

Hence, there exists a $T(\epsilon,M)$ so that for all $t_n>T(\epsilon,M)$
$$\displaystyle
d_{Y_{\rho}}(S^z_{t_n}(y^M_0),\widehat{y^z}) < \dfrac{\epsilon}{2}.$$
Then
\begin{align}
d_{Y_{\rho}}(S^z_{t_n}(y_0),\widehat{y^z}) \le &~ d_{Y_{\rho}}(S^z_{t_n}(y_0),S^z_{t_n}(y^M_0))+d_{Y_{\rho}}(S^z_{t_n}(y^M_0),\widehat{y^z}) < ~\epsilon.
\end{align}
Owing to the decay of $(u,u_t)$ in Lemma \ref{expattract}, we note that, in fact, $(\widehat{z},\widehat{\phi^z}) \in \mathcal N^{z,*}$. This implies the desired convergence result, and concludes the proof of Theorem \ref{dichotoo1}. \end{proof}
\subsection{ $  (w,\phi^w) $ Dynamics }

Now we consider the dynamics given by $S^w_t(\cdot)$ on $Y_{\rho}$ and note the regularity of the dynamics (see Theorem \ref{regattract}):
\begin{theorem}\label{dichotoo2} Assume $y_0 \in Y_{\rho_0}$. 
For $k>k_Q$, the dynamics $S_t^w(\cdot)$ have the property that for any sequence $t_n \to \infty$ there is a subsequence $t_{n_j}$ and a point $\widehat{y^w}\equiv (\widehat w,0;\widehat{\phi^w},0)$ with $(\widehat w,\widehat{\phi^w}) \in \mathcal N$ so that
$$\ds \lim_{j \to \infty} ||S_{t_{n_j}}^w(y_0)-\widehat{y^w}||_{Y_{\rho}} = 0$$~~ for any $\forall~\rho>0$. This implies the desired convergence result for $S_t^w(\cdot)$ to the set $\mathcal N$
\end{theorem}
\begin{proof}[Proof of Theorem \ref{dichotoo2}]
Consider the plate dynamics $T_t^w(\cdot)$ corresponding to: 
\begin{equation}\label{smooth**}\begin{cases} w_{tt} + (k+1)w_t + \Delta^2 w  = p_0+q^w(t)-Uw_x-[b-||\nabla u||^2]\Delta w+\beta z~~ \text{ in } \Omega\\
w=\Dn w=0~~\text{ on } \Gamma\\
w(t_0) = 0, ~~w_t(t_0) = 0,~~w^{t_0}=0.\end{cases}\end{equation}

Lemma \ref{regattract} guarantees a bound in higher norms  $ H^r$ for the $(w,w_t;w^t)$ dynamics, satisfying \eqref{thisone} in Theorem \ref{highernorms0}. The analysis of Theorem \ref{highernorms0} obtains the equivalent of Theorem \ref{regresult} for the $w$ dynamics in \eqref{wfull}. Here, we critically use the fact that the $z$ terms above decay along any sequence of times $t_n \to \infty$, i.e., along any sequence of times $t_n \to \infty$ we have that $z(t_{n}) \to 0$ in $H^2(\Omega)$ and $z_t(t_{n}) \to 0$ in $L_2(\Omega)$.  Considering the estimates
$$
||\beta z||_0 \le  ~C||z||_2;\hskip.1cm~~[b-||\nabla u||^2]\Delta w = ~ [b-||\nabla w||^2]\Delta w-||\nabla z||^2\Delta w,$$ and the aforementioned exponential decay, we see that $S^w_{t_n}(y_0) \to (\widehat w, 0; \widehat{\phi^w}, 0)$ (with $(\widehat w,\widehat{\phi^w}) \in \mathcal N$) in $Y_{\rho}$ for any $\rho>0$. This concludes the proof of Theorem \ref{dichotoo2}.
\end{proof}

\subsection{Final Step in the Proof of Theorem \ref{conequil}}
Note the fact that for $(\widehat z,\widehat{\phi^z}) \in \mathcal N^{z,*}$ and $(\widehat w,\widehat{\phi^w}) \in \mathcal N$, we have that \begin{equation}(\widehat z,\widehat{\phi^z})+(\widehat w,\widehat{\phi^w}) \in \mathcal N\end{equation} (due the linear structure of the stationary flow equation \eqref{static}). By Theorem \ref{dichotoo1} and Theorem \ref{dichotoo2}, for any sequence $t_n \to \infty$, there is a subsequence (also labeled) $t_n$ and two points $\widehat{y^z} =(\widehat z,0;\widehat{\phi^z},0)$ (with $(\widehat z, \widehat{\phi^z}) \in \mathcal N^{z,*}$) and $\widehat{y^w} =(\widehat w,0;\widehat{\phi^w},0)$ (with $(\widehat w, \widehat{\phi^w}) \in \mathcal N$) so that $S^z_{t_n}(y_0) \to (\widehat z,0;\widehat{\phi^z},0)$ and $S^w_{t_n}(y_0) \to (\widehat w,0;\widehat{\phi^w},0)$ in $Y_{\rho}$ for any $\rho>0$. Thus
 \begin{align}
||S_{t_n}(y_0)-(\widehat{y^z}+\widehat{y^w})||_{Y_{\rho}} =&~||(S^z_{t_n}(y_0)-\widehat{y^z})+(S^w_{t_n}(y_0)-\widehat{y^w})||_{Y_{\rho}}\\
 \le&~||S^z_{t_n}(y_0)-\widehat{y^z}||_{Y_{\rho}}+||S^w_{t_n}(y_0)-\widehat{y^w})||_{Y_{\rho}}
 \to 0,~\text{ as }~n\to\infty.
 \end{align}
Since $(\widehat z;\widehat{\phi^z})+(\widehat w;\widehat{\phi^w}) \in \mathcal N$, this implies the desired strong convergence result and completes the proof of Theorem \ref{conequil}.

\section{Open Problem: Eliminating Damping Size Requirements} \label{openprob}
In experimental and numerical simulations with panels that have no imposed damping,  flutter does not occur for subsonic flows (see Section \ref{prevv}). The techniques herein (and in \cite{chuey,conequil1,ryz,ryz2}) certainly require {\em some} damping. However, for finite energy solutions we require a damping coefficient whose size depends on the {\em inherent parameters of the problem} in order to show convergence of full trajectories to the set of stationary states. We note that there is viscous damping in every physical plate system, but we do question whether the dependence of the results above on a minimal damping coefficient $k_{min}$ is necessary. (Especially since there is no size restriction for ``smooth" initial data---Theorem \ref{regresult}---even for von Karman dynamics.) It seems that in a correct model the requirement could be removed in order to yield agreement with experiment, OR that there is some physical situation where the structural damping is sufficiently small that the subsonic panel may exhibit non-stationary end behavior. 
The authors could find no discussion of the latter situation in the engineering literature.

	In our work above, the key issue necessitating large damping is the interaction of the terms $Lu$ and $q(u^t)$ with the multiplier analysis in Section \ref{expsec} and Section \ref{regsec}. The following intermediate assertion is presently clear:
	Taking $b = 0$, $p_0\equiv 0$, $U<<1$, will not eliminate the need for a damping restriction to obtain strong stabilization of the full flow-plate dynamics due to the non-conservative, non-dissipative contributions from $q$. An interesting question is: under what restrictions (of the form $b \equiv 0$ and $p_0\equiv 0$, or at least small in some sense) can we expect {\em stability} of the plate dynamics---no non-trivial stationary states? From what is known physically, large damping, diminished loading, {\em or} small unperturbed flow velocities should lead to a simplification of the stationary set $\mathcal N$.

\section{Acknowledgements} 
The authors would like to dedicate this work to the memory of Professor A.V. Balakrishnan.

The authors sincerely thank professor E.H. Dowell of Duke University's Mechanical Engineering Department for his assistance, and for his insightful (and immensely helpful) commentary since 2010.

\end{document}